\documentclass[bj]{imsart}
\setattribute{journal}{name}{}
\usepackage[utf8]{inputenc}
\usepackage[colorlinks,citecolor=blue,urlcolor=blue]{hyperref}

\usepackage[abs]{overpic}

\usepackage[authoryear]{natbib}
\usepackage{times}

\usepackage{multirow}
\usepackage{booktabs}
\usepackage{amssymb,amsmath, amsthm}
\usepackage[nobbm]{optional}
\usepackage{graphicx}
\usepackage{color}
\opt{bbm}{
\usepackage{bbm} 
}
\usepackage{framed} 

\newtheorem{thm}{Theorem}
\newtheorem{lemma}{Lemma}
\newtheorem{cor}{Corollary}
\theoremstyle{definition}
\newtheorem{algorithm}{Algorithm}
\newtheorem{prop}{Proposition}
\newtheorem{rem}{Remark}

\newtheorem{ass}{Assumption}

\providecommand{\NN}{\mathbb{N}}

\providecommand{\RR}{\mathbb{R}}

\providecommand{\I}{\mathrm{Id}}

\providecommand{\E}{\mathbb{E}}
\providecommand{\e}{\mathrm{e}}
\renewcommand{\P}{\mathbb{P}} 

\providecommand{\abs}[1]{\left\lvert#1\right\rvert}

\providecommand{\norm}[1]{\lVert#1\rVert}
\providecommand{\ind}[1]{{\mathbf{1}_{#1}}}
\opt{bbm}{
\renewcommand{\ind}[1]{{\mathbbm{1}_{#1}}}
}

\newcommand{\T}{{\prime}}

\newcommand{\dhp}{{\!\triangledown}}
\newcommand{\mbb}{{\!\vartriangle}}

\providecommand{\argmin}[0]{{\operatorname{arg min}}}

\providecommand{\trace}{{\operatorname{tr}}}

\newcommand{\dd}{{\,\mathrm d}}
\newcommand{\DD}{{\,\mathrm D}}

\newcommand{\EE}{\mathbb{E}}

\renewcommand{\tilde}{\widetilde}
\renewcommand{\theta}{\vartheta}
\renewcommand{\phi}{\varphi}
\newcommand{\eps}{\varepsilon}

\newcommand{\qv}[1]{\left<{#1}\right>}
\newcommand{\pin}[1]{{#1}^\star}

\newcommand{\Ell}{{\mathcal L}}

\newcommand{\rara}[1]{\renewcommand{\arraystretch}{#1}}

\newcommand{\si}{\sigma}

\renewcommand{\th}{\theta}
\newcommand{\la}{\lambda}
\newcommand{\ga}{\gamma}

\newcommand{\expec}[1]{\ensuremath{{\mathbb E}\mspace{-1mu}\left[#1\right]}}

\begin{document}

\begin{frontmatter}

\title{Guided proposals for simulating multi-dimensional diffusion bridges}
\runtitle{Guided proposals for multi-dimensional diffusion bridges}


\begin{aug}
\author{\fnms{Moritz} \snm{Schauer}, \ead[label=e1]{m.r.schauer@tudelft.nl}
\fnms{Frank} \snm{van der Meulen} \ead[label=e2]{f.h.vandermeulen@tudelft.nl}\\
\and
\fnms{Harry} \snm{van Zanten}\ead[label=e3]{hvzanten@uva.nl}}

\thankstext{}{Research supported by the Netherlands Organization for Scientific Research (NWO)}

\runauthor{Schauer, Van der Meulen and Van Zanten}

\affiliation{Delft University of Technology and University of Amsterdam}

\address{
Delft Institute of Applied Mathematics (DIAM) \\
Delft University of Technology\\
Mekelweg 4\\
2628 CD Delft\\
The Netherlands\\
\printead{e1}\\
\printead{e2}}

\address{Korteweg-de Vries Institute for Mathematics\\
University of Amsterdam\\
P.O. Box  94248\\
1090 GE Amsterdam\\
The Netherlands\\
\printead{e3}}

\end{aug}

\begin{abstract}
A Monte Carlo method for simulating a multi-dimensional diffusion process conditioned on hitting a fixed point at a 
fixed future time is developed. 
Proposals for such diffusion bridges are obtained by superimposing an additional guiding term to the drift of the process under consideration. The guiding term is derived via approximation of the target process by a simpler diffusion processes with known transition densities. Acceptance of a proposal can be determined by computing the likelihood ratio between
the proposal and the target bridge, which is derived in closed form. We show under general conditions that the likelihood ratio is well defined
 and show that a class of proposals with guiding term obtained from linear approximations fall under these conditions.

\medskip

\noindent
 \emph{ {Keywords:} Multidimensional diffusion bridge; change of measure; data augmentation; linear processes.}
 \end{abstract}

\begin{keyword}[class=MSC]
\kwd[Primary ]{60J60}
\kwd[; secondary ]{65C30}
\kwd{65C05}
\end{keyword}

\end{frontmatter}

\numberwithin{equation}{section}

\section{Introduction}

\subsection{Diffusion bridges}

Suppose $X$ is a $d$-dimensional diffusion with time dependent drift $b\colon \RR_+ \times \RR^d \rightarrow \RR^d$ and  
dispersion coefficient $\sigma\colon \RR_+ \times \RR^d \rightarrow \RR^{d\times d'}$ governed by the stochastic differential equation (SDE)
\begin{equation}\label{sde}
\dd X_t = b(t, X_t) \dd t + \sigma(t, X_t) \dd W_t, \quad X_0 = u,
\end{equation}
where $W$ is a standard ${d'}$-dimensional Brownian motion.
When the process $X$ is conditioned to hit a point $v \in \RR^d$ at time $T > 0$, 
the resulting process $X^\star$ on $[0,T]$ is called the {\em diffusion bridge from $u$ to $v$}. 
In this paper we consider the problem of simulating realizations of this bridge process. 
Since we are conditioning on an event of probability zero and in general 
 no closed form expression for the transition densities  of the  original process $X$ or the bridge $X^\star$ exist, 
 this is known to be a difficult problem. 

This  problem arises for instance when making statistical inference for diffusion models 
from discrete-time, low-frequency data. In that setting the fact that the transition densities 
are unavailable implies that the likelihood of the data is not accessible. 
A successful approach initiated by  \cite{RobertsStramer} is to  circumvent this problem by viewing the 
continuous segments between the observed data points  as missing data.
Computational algorithms can then be designed that  augment the discrete-time data by 
(repeatedly) simulating the diffusion bridges between the observed data points. 
This statistical application of simulation algorithms for diffusion bridges 
was our initial motivation for this work. The present paper however focusses  on 
the simulation problem as such and can have other  applications as well.

The 
simulation of diffusion bridges has received much attention over the past decade, see for instance 
the papers \cite{ElerianChibShephard},  \cite{Eraker},  \cite{RobertsStramer}, 
\cite{DurhamGallant}, \cite{Stuart}, \cite{BeskosRoberts}, \cite{BeskosPapaspiliopoulosRoberts}, \cite{Voss}, \cite{Fearnhead},
 \cite{PapaspiliopoulosRoberts}, \cite{Mykland}, \cite{Bladt}, \cite{Schoenmakers}  to mention just a few. Many of these papers employ accept-reject-type methods. 
 The common idea is that while sampling directly from the law $\P^\star$ of the bridge process $X^\star$ 
 is typically impossible, sampling from an equivalent law $\P^\circ$ of some proposal process
$X^\circ$ might in fact be feasible. If this proposal is accepted with an appropriately chosen probability, 
depending on the Radon-Nikodym derivative $(\rm{d}\P^\star/\rm{d}\P^\circ)(X^\circ)$, then 
either exact or approximate draws from the target distribution $\P^\star$ can be generated. 
Importance sampling and Metropolis-Hastings algorithms  are the prime examples of 
methods of this type. 

To be able to carry out these procedures in practice, simulating paths from the proposal process
has to be relatively easy and, up to a normalizing constant, an expression for the derivative 
$(\rm{d}\P^\star/\rm{d}\P^\circ)(X^\circ)$ has to be available that is easy to evaluate.
The speed of the procedures greatly depends on the acceptance probability, 
which in turn  depends on $(\rm{d}\P^\star/\rm{d}\P^\circ)(X^\circ)$. 
This can be influenced by working with a cleverly chosen proposal process $X^\circ$.
A naive choice might result in a proposal process that, although its law 
is equivalent to that of the target bridge $X^\star$, has sample paths that are with 
considerable probability  rather different from those of $X^\star$. This then results 
in  small ratios  $(\rm{d}\P^\star/\rm{d}\P^\circ)(X^\circ)$ with large probability, which in turn  
leads to small acceptance probabilities and hence to a slow procedure. 
It is therefore desirable to have proposals that are 
``close'' to the target in an appropriate sense. In this paper we 
construct such proposals for the multi-dimensional setting.

\subsection{Guided proposals}

We will consider so-called {\em guided proposals}, according to the terminology 
 suggested in  \cite{PapaspiliopoulosRoberts}. This means that our proposals 
 are realizations of a process $X^\circ$ that solves an SDE of the form 
 \eqref{sde} as well, but with a drift term that is adapted in order to 
 force the process $X^\circ$ to hit the point $v$ at time $T$.

An  early paper suggesting guided proposals is  \cite{Clark} (a paper that seems to have 
received little attention in the statistics community). \cite{Clark} considers the case $d = 1$ and $\sigma$ constant and advocates using 
proposals from the SDE $ \dd X^\circ_t =b(X^\circ_t)\dd t+ \tfrac{v-X^\circ_t}{T-t}\dd t + \si \dd W_t$. 
Note that here the guiding drift term that drives the process to $v$ at time $T$ is exactly the drift  term of a Brownian bridge.
In addition  the drift $b$ of the original process appears. The  idea 
is that this ensures that before time $T$, the 
proposal behaves similar to the original diffusion $X$.
 \cite{DelyonHu} have generalized the work of \cite{Clark} 
 in two important directions. Firstly, they allow
 non-constant  $\si$ using proposals $X^\dhp$ satisfying the SDE  
\begin{equation}\label{xcircp}\tag{$\dhp$}
 \dd X^\dhp_t =\left(b(t,X^\dhp_t)+ \frac{v-X^\dhp_t}{T-t}\right)\dd t + \si(t, X^\dhp_t) \dd W_t. 
\end{equation}
This considerably complicates proving that the laws of $X^\circ$ and the target bridge 
 $X^\star$ are absolutely continuous. Further, 
\cite{DelyonHu} consider the alternative proposals $X^\mbb$ satisfying the SDE 
 \begin{equation*}\label{xcircdp}\tag{$\mbb$}
  \dd X^\mbb_t = \frac{v-X^\mbb_t}{T-t}\dd t + \si(t, X^\mbb_t) \dd W_t. 
 \end{equation*}
where the original drift of $X$ is disregarded. This is a popular choice in practice especially with a discretization scheme known as the Modified Brownian Bridge. Both proposals have their individual drawbacks, see Section \ref{sec:comp}.

Another important difference is that they consider the multi-dimensional case.
With more degrees of freedom a proposal process that is not appropriately chosen has a much higher chance of  
not being similar to the  target process, leading to very low acceptance probabilities and 
hence slow simulation procedures. In higher dimensions the careful construction of the proposals is 
even more important  for obtaining  practically feasible procedures than in dimension one.

Our approach is inspired by the ideas in \cite{Clark} and \cite{DelyonHu}.  However, 
we propose to adjust the drift in a different way,  allowing more flexibility in constructing an appropriate guiding term.
This is particularly aimed at finding procedures with higher acceptance probabilities
in the multi-dimensional case. 
To explain the approach in more detail we recall that, under weak assumptions 
the target diffusion bridge $X^\star$ is characterized as the solution to the SDE
\begin{equation}\label{xstar} \tag{$\star$}
   \dd X^\star_t=   b^\star(t,X^\star_t) \dd t + \si(t, X^\star_t) \dd W_t,\qquad X^\star_0=u,\qquad t \in [0,T),
\end{equation}
where 
\begin{equation}\label{bstar} \tag{$\star\star$}	
b^\star(t,x) =b(t,x)+ a(t,x) \nabla_x \log p(t,x;T,v)
\end{equation}
and $a(t,x) = \sigma(t,x)\sigma^\T(t,x)$.
In the bridge SDE the term $ a(t,x) \nabla_x \log p(t,x;T,v)$  is added to the original drift  to direct $X^\star$ 
towards $v$ from the current position $X^\star_t = x$ in just the right manner.
Since equation \eqref{xstar} contains the unknown transition densities of the original process $X$ it cannot be employed 
directly for simulation. We propose to replace this unknown density  by one coming from  
an auxiliary  diffusion process with known transition densities. So the proposal process is going to be 
the solution  $X^\circ$ of the SDE
\begin{equation}\label{xcirc} \tag{$\circ$}
\dd X^\circ_t = b^\circ(t, X^\circ_t) \dd t + \sigma(t, X^\circ_t) \dd W_t, \quad X^\circ_0 = u,
\end{equation}
where
\begin{equation}\label{bcirc} \tag{$\circ\circ$}
			b^\circ(t, x) = b(t,x) + a(t, x) \nabla_x  \log  \tilde p(t,x;T,v)
\end{equation}
and   $\tilde p(s,x; t, v)$  is  the transition density of a diffusion process $\tilde{X}$  for which above expression 
is known in closed form. 
We note that in general our proposals are different from those defined in \cite{DelyonHu}. First of 
all  the diffusion  $a(t,x)$ of the original process appears  in the drift of the proposal process  $X^\circ$
and secondly we have additional freedom since we can  choose the process $\tilde{X}$.

The paper contains two main theoretical results. In the first we give conditions under which 
the process $X^\circ$ is indeed a valid proposal process in the sense that 
its distribution $\P^\circ$ (viewed as Borel measure on $C([0,T], \RR^d)$) is equivalent to 
the law $\P^\star$ of  the target process $X^\star$ and we derive 
an expression for the Radon-Nikodym derivative of the form
\[
\frac{\dd \P^\star}{\dd \P^\circ}({X^\circ}) \propto
 \exp\left( \int_0^T G(s,X^\circ_s) \dd s\right),
 \]
where the functional $G$ does not depend on unknown or inaccessible objects. 
In the second theorem we show that the  assumptions of the general result are fulfilled if 
in (\ref{bcirc}) we 
choose the transition density $\tilde p$ of a process $\tilde{X}$ from a large  class of 
{\em linear processes}. This is a  suitable class,  since linear processes have  tractable transition densities.

\subsection{Comparison of proposals}\label{sec:comp}
 Numerical experiments presented \cite{vdMeulenSchauer} show that our approach  can  indeed substantially increase acceptance rates in a Metropolis-Hastings sampler, 
 especially in the multi-dimensional setting. Already in a simple one-dimensional example however we can illustrate the advantage of our method. 

Consider the solution $X$ of the SDE, 
\[	 
\dd X_t = b(X_t) \dd t + \tfrac12\dd W_t, \quad X_0=u
\quad \text{with}\quad b(x) = \beta_1 - \beta_2 \sin(8x).
\]
The corresponding bridge $X^\star$ is obtained by  conditioning $X$ to hit the point $v \in \RR$ at time $T> 0$. We take $u=0, v=\tfrac \pi2$ and consider either the case  $\beta_1 = \beta_2 = 2$ or $\beta_1 = 2, \beta_2 = 0$.
We want to compare the three mentioned proposals \eqref{xcircp},\eqref{xcircdp} and  \eqref{xcirc} in these two settings. The drift $b$ satisfies the assumptions for applying the Exact Algorithm of \cite{BeskosRoberts}, but numerical experiments revealed the rejection probability is close to $1$ in this particular example. Besides, our main interest lies in comparing proposals that are suited for simulating general diffusion bridges in the multivariate case as well. 
A simple choice for the guided proposal \eqref{xcirc} is obtained by taking $\tilde X$ to be a scaled Brownian motion with constant drift $\theta$. This gives
 $b^\circ(s,x) = b(x) + \frac{v - x}{T-s} - \theta$ as
  the drift of the corresponding guided proposal. Here we can  choose $\theta$ freely. In fact, far more flexibility can be obtained by choosing  $\tilde{X}$ a linear process as in theorem \ref{thm:linearproc}. In particular, we could take $\th$ to depend on $t$, resulting in an infinite dimensional class of proposals. For illustration purposes, in this example we show that just taking a scaled Brownian motion with constant drift $\theta$ for $\tilde X$ is already very powerful. 
  
If $\beta_2 = 0$ the process $X$ is simply a Brownian motion with drift. It is folklore that the corresponding bridge $X^\star$ is then in fact the standard Brownian bridge from $u$ to $v$, 
independent of the constant $\beta_1$ (see for instance \cite{Esko}). So in that 
case both proposal \eqref{xcircdp} and proposal \eqref{xcirc} with $\theta = \beta_1$ coincide with the target bridge. However,   the drift $b^{\dhp}$  of the proposal \eqref{xcircp} is off by $|b^\star(s,x) - b^{\dhp}(s,x)| = |\beta_1|$ leading to bad acceptance rates if $\beta_1 \neq 0$, even for small values of $T$. This seems to be the prime reason that proposal \eqref{xcircp} is rarely used in practice. 

Now if $\beta_2 = 2$, both \eqref{xcircp} and \eqref{xcircdp} fail to capture the true dynamics of \eqref{xstar}. Roughly speaking, for  \eqref{xcircdp} the proposals fail to capture the multimodality of the marginal distributions of the true bridge, while proposals with  \eqref{xcircp}  arrive at values close to $v$ too early due to the mismatch between pulling term and drift. On the other hand the proposals \eqref{xcirc} can be quite close to the target bridge for good choices of $\theta$, see figure \ref{fig:sinexample2}. Two effects are in place: incorporating the true drift into the proposal results in the correct local behaviour of the proposal bridge (multimodality in this particular example).  Further, an appropriate choice of $\th$ reduces the mismatch between the drift part and guiding part of the proposal.
\begin{figure}[bth]
\centering{\includegraphics[width=.85\linewidth]{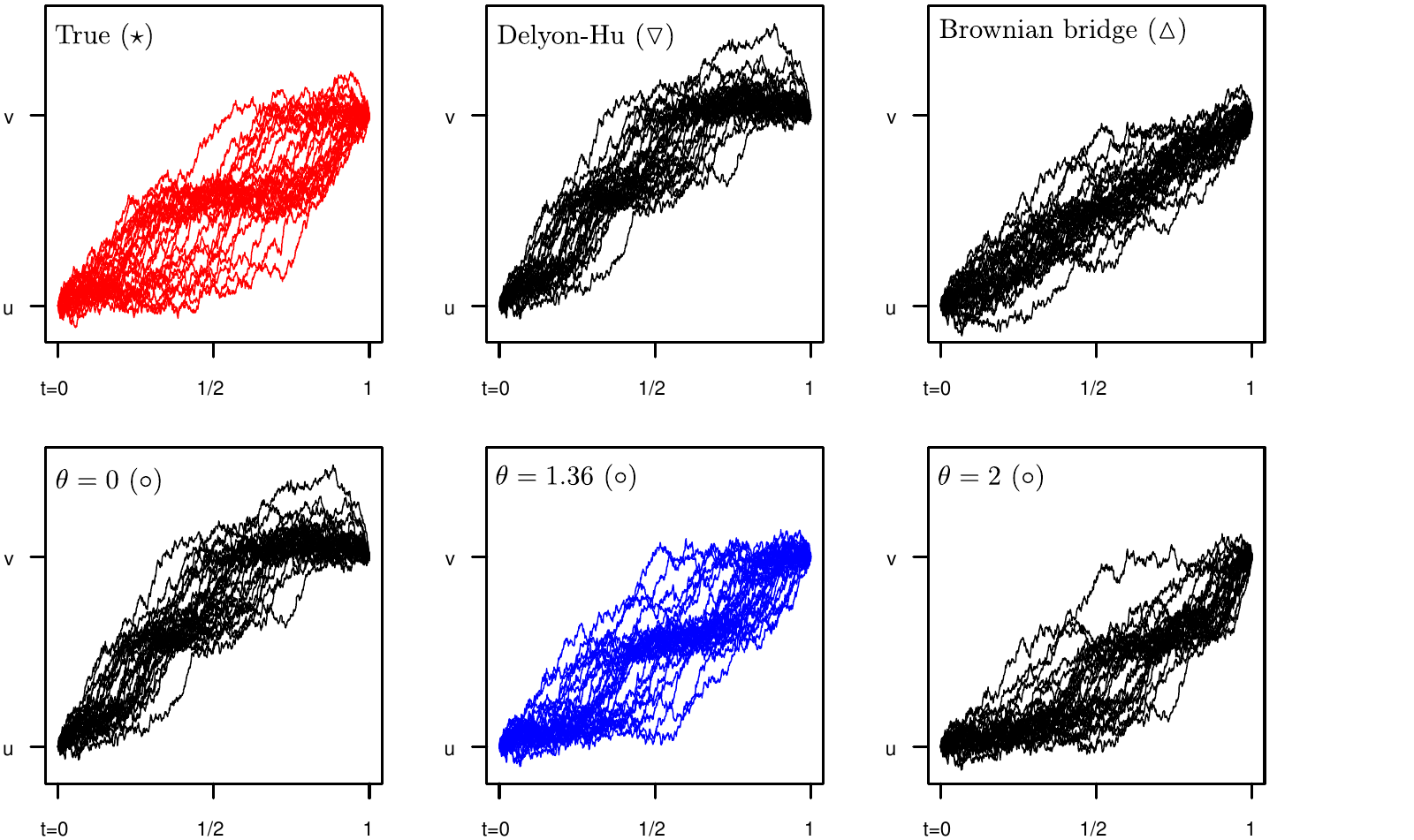}}
\caption{Samples from the true distribution of the bridge compared to different proposals for the example $b(x) = 2 - 2\sin(8x)$. Top row: True bridge, proposals with drift  $b^{\dhp}(t, x) = b(x) + \frac{v-x}{T-t}$ and $b^{\mbb}(t, x) =   \frac{v-x}{T-t}$. Bottom row: $b^\circ(s, x) = b(x) + \frac{v-x}{T-t} - \theta$ for different values of $\theta$. The top-middle figure and bottom-left figure coincide.}\label{fig:sinexample2}
\end{figure}
The additional freedom in \eqref{xcirc} by choice of $\th$ will be especially useful, if one can find good values for $\th$ in a systematic way. We now explain how this can be accomplished.

Let $P^\circ_\theta$ denote the  law of $X^\circ$.
One option to choose $\th$ in a systematic way is to take the information projection $P^\circ_{\theta_{\text{opt}}}$ defined by 
\[\theta_{\text{opt}} = \argmin_\th D_{\operatorname{KL}}( \P^\star  \| \P^\circ_\theta) \]
Here, the Kullback-Leibler divergence is given by 
\[D_{\operatorname{KL}}(\P^\star \| \P^\circ_\theta) = \int \log \left(\frac{\dd \P^\star}{\dd \P^\circ_\theta}\right)  \dd \P^\star.\]
This is a measure how much information is lost, when $P^\circ_\theta$ is used to approximate $P^\star$. This expression is not of much direct use, as it depends on the unknown measure $P^\star$. However, 
given a sample $X^\circ$ from $\P^\circ_{\th_0}$ using a reference parameter $\th_0$, the gradient of $D_{\operatorname{KL}}( \P^\star \| \P^\circ_\theta)$ can be approximated by 
\[ \nabla_\theta  \log \frac{\dd \P^\star}{\dd \P^\circ_\theta}(X^\circ)   \frac{\dd \P^\star}{\dd \P^\circ_{\theta_0}}(X^\circ).  \]
This in turn can be used in an  iterative stochastic gradient descent algorithm (details are given in the appendix).
The value $\th=1.36$ used in Figure \ref{fig:sinexample2} was obtained in this way. From the trace plot  of the gradient descent algorithm displayed in figure \ref{fig:sinexample} it appears the algorithm settles near the optimal value shown in the right-hand figure. 
\begin{figure}[bth]
\centering{\includegraphics[width=.45\linewidth]{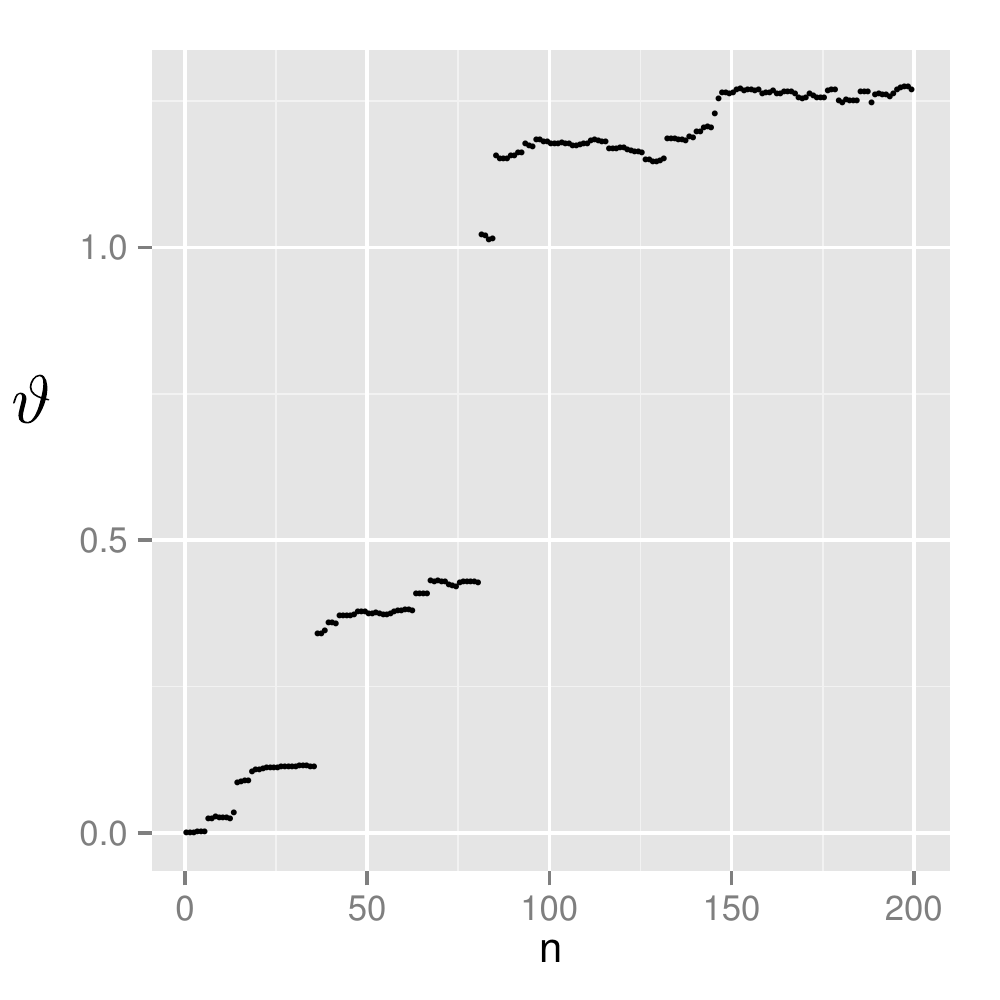}\includegraphics[width=.45\linewidth]{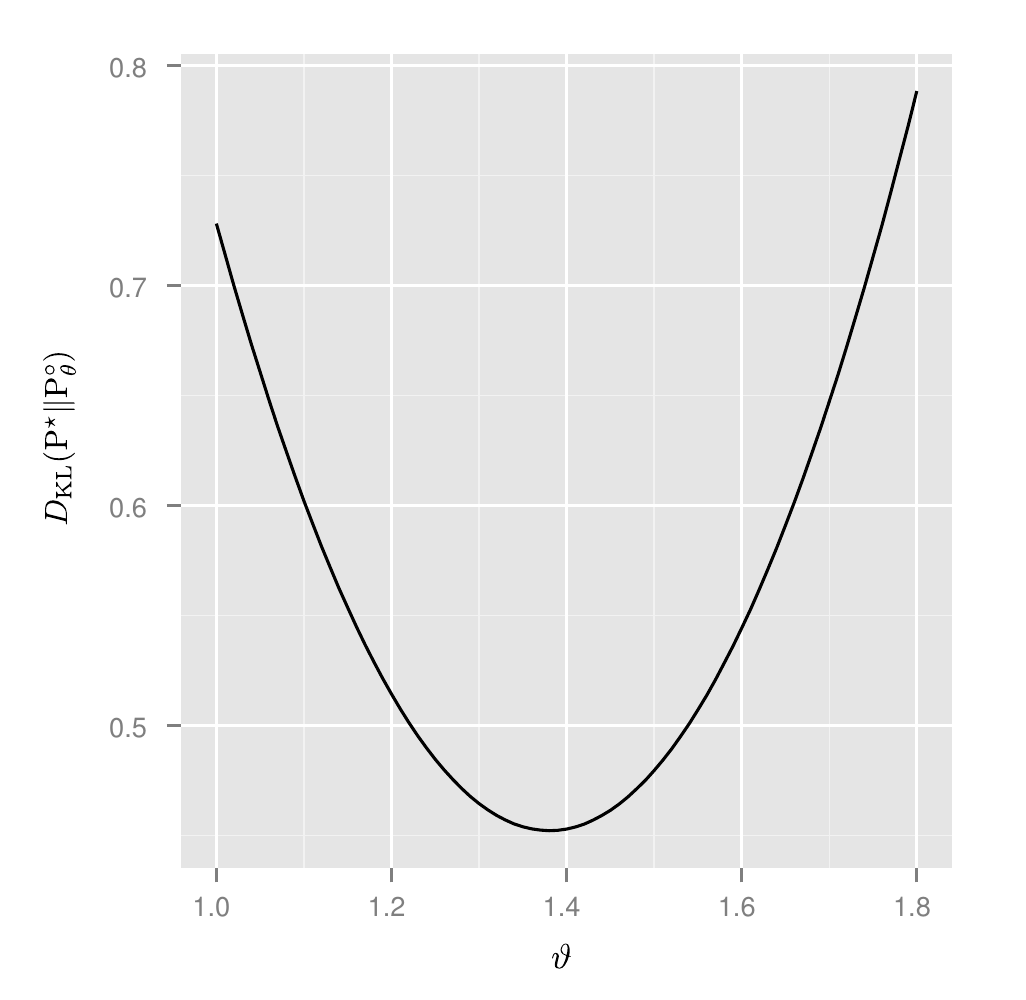}}
\caption{Left: trace plot of $\theta$ using the stochastic gradient descent algorithm. Right: $\th \mapsto D_{\operatorname{KL}}(\P^\star \| \P^\circ_\theta)$, estimated with $100000$ simulated bridges.}\label{fig:sinexample}
\end{figure}

\subsection{Contribution of this paper}
In this paper we propose a novel class of proposals for generating diffusion bridges that can be used in Markov Chain Monte Carlo and importance sampling algorithms. We stress that these are not special cases of the proposals from \cite{DelyonHu} (specified in equations \eqref{xcircp} and \eqref{xcircdp}). An advantage of this class is that the drift of the true diffusion process is taken into account while avoiding the drawbacks of proposals of the form \eqref{xcircp}. This is enabled by the increased flexibility for constructing a pulling term in the drift of the proposal.  A particular feature of our choice is that no It\=o-integral appears in the likelihood ratio between the true bridge and proposal process. Furthermore, the dispersion coefficient $\si$ does not need to be invertible. In a companion paper (\cite{vdMeulenSchauer}) we show how guided proposals can be used for Bayesian estimation of discretely observed diffusions.

\subsection{Organization}
The main results of the paper are presented in Section \ref{sec: main}. proofs are given in Sections \ref{sec:proof_theorem}--\ref{sec:proof_linearproc}.

\subsection{General notations and conventions}\label{sec:notation}

\subsubsection{Vector- and matrix norms}

The transpose of a matrix $A$ is denoted by $A^\T$.  The determinant and trace of a square matrix $A$ are denoted by $|A|$ and $\trace(A)$ respectively. 
For vectors, we will always use the Euclidean norm, which we denote  by $\|x\|$. For a $d\times d'$ matrix $A$, we denote its Frobenius norm by $\|A\|_F=(\sum_{i=1}^d \sum_{j=1}^{d'} A^2_{ij})^{1/2}$.
The spectral norm, the operator norm induced by the Euclidean norm will de denoted by $\|A\|$, so
\[ 	\|A\|= \sup\{ \|Ax\|,\: x \in \RR^{d'} \:\text{with}\: \|x\|=1\}. \]
Both norms are submultiplicative, $\|A x\| \le \|A\|_F \|x\|$ and  $\|A x\| \le \|A\| \|x\|$. The identity matrix will be denoted by $\I$.

\subsubsection{Derivatives}
For $f\colon \RR^{m} \to \RR^n$ we denote by $\DD f$ the $m\times n$-matrix with element $(i,j)$ given by 
$\DD_{ij} f(x)= ({\partial f_j}/{\partial x_i})(x)$.
If  $n=1$, then   $\DD f$ is the column vector containing all partial derivatives of $f$, that is $\nabla_x f$ from the first section. In this setting we write the $i$-th element of $\DD f$ by $\DD_i f(x)= ({\partial f}/{\partial x_i})(x)$ and  denote $\DD^2 f=\DD(\DD f)$ so that $\DD^2_{ij} f(x) =
{\partial^2 f(x)}/({\partial x_i\partial x_j})$.
If $x\in \RR^n$ and  $A \in \RR^{n\times n}$ does not depend on $x$, then $\DD (Ax)=A^\T$. Further, for $f\colon \RR^n \to \RR^n$ we have
\[ \DD(f(x)^\T A f(x))=  (\DD f(x))^\T (A+A^\T)  f(x). \]
Derivatives with respect to time are always denoted as $\partial / \partial t$.

\subsubsection{Inequalities}
We write $x \lesssim y$ to denote that there is a universal (deterministic) constant $C>0$ such that $x \le C y$.


\section{Main results}\label{sec: main}

\subsection{Setup}
We continue to use the notation of the introduction, 
so the process $X$ is the unconditioned process
defined as the solution to the SDE \eqref{sde}. We assume 
throughout that the 
functions $b$ and $\sigma$ are Lipschitz in both arguments, 
satisfy a linear growth condition in their second argument and that $\sigma$
is uniformly bounded. These conditions imply in particular that 
the SDE has a unique strong solution (e.g.\ \cite{Karatzas-Shreve}).
The auxiliary process $\tilde X$ whose transition densities are 
used in the proposal process is defined as the solution of 
an SDE like \eqref{sde} as well, but with drift $\tilde b$
instead of $b$ and  dispersion $\tilde\sigma$ instead of $\sigma$. The functions 
$\tilde b$ and $\tilde\sigma$ are assumed to satisfy the same Lipschitz,  linear 
growth and boundedness conditions as $b$ and $\sigma$. We write $a = \sigma\sigma^\T$ and 
$\tilde a = \tilde\sigma\tilde\sigma^\T$.

The processes $X$ and $\tilde X$ are assumed to have smooth
transition densities with respect to Lebesgue measure. 
More precisely, denoting the law of the process $X$ started in $x$ at time $s$ by $\P^{(s,x)}$, we assume that that for $0\le s<t$ and $y \in \RR^d$
\[	
\P^{(s,x)}(X_t \in \dd y)= p(s, x; t,y) \dd y
\]
and similarly for the process $\tilde X$, whose transition densities are denoted by $\tilde p$
instead of $p$. 
The infinitesimal generators of $X$ and $\tilde X$ 
are denoted by   $\Ell$ and $\tilde\Ell$, respectively, so that
\begin{equation}\label{infgen}
(\Ell f) (s, x)  = \sum_{i=1}^{d} b_i(s,x) \DD_i f(s,x) + \frac12\sum_{i,j=1}^{d}  a_{ij}(s,x) \DD^2_{ij} f(s,x),
\end{equation}
for  $f \in C^{1,2}(\RR \times \RR^d, \RR)$, and similarly for $\tilde \Ell$ (with $\tilde b$ and $\tilde a$). 
Under regularity conditions, which we assume to be fulfilled, we have that the transition densities  of $\tilde{X}$ satisfy 
{Kolmogorov's backward equation}:
\[ 
\frac{\partial}{\partial s} \tilde{p}(s,x; t, y)+ (\tilde\Ell \tilde{p})(s,x; t, y)=0 
\]
(here $\tilde\Ell$ acts on $s,x$). 
(See for instance \cite{Karatzas-Shreve}, p. 368, for sufficient regularity conditions.)

We fix a time horizon $T > 0$ and a point $v \in \RR^d$ such that for all $s \le T$ and $x \in \RR^d$ it holds that 
$p(s, x; T, v) > 0$ and $\tilde p(s, x; T, v) > 0$. The target bridge process $X^\star = (X^\star_t: t \in [0,T])$ is defined 
by conditioning the original process $X$ to hit the point $v$ at time $T$. 
The proposal process $X^\circ = (X^\circ_t: t \in [0,T])$ is defined as the solution 
of \eqref{xcirc}--\eqref{bcirc}. In the results ahead we will impose conditions on the transition densities 
$\tilde p$ of $\tilde X$ that imply that this SDE has a unique solution. 
All processes are assumed to be defined on the canonical path space and 
$(\mathcal F_t)$ is the corresponding canonical filtration. 

For easy reference, the following table 
briefly describes the various processes around.
\begin{center} \rara{1.2}
\begin{tabular}{|l|l|}
\hline
$X$ & original, unconditioned diffusion process\\
$X^\star$ & corresponding bridge, conditioned to hit $v$ at time $T$, defined through \eqref{xstar}\\
$X^\circ$ & proposal process defined through \eqref{xcirc}\\ 
$\tilde X$ & auxiliary process whose transition densities $\tilde p$ 
appear in the definition of $X^\circ$.\\
\hline
\end{tabular}
\end{center}
We denote the laws of $X$, $X^\star$ and $X^\circ$ viewed as measures on the space $C^d([0,t], \RR^d)$ of continuous functions from $[0,t]$ to $\RR^d$ equipped with Borel-$\sigma$-algebra by  $\P_t$, $\P^\star_t$ and $\P^\circ_t$ respectively. In case $t=T$ we drop the subscript $T$.

\subsection{Main results}

The end-time $T$ and the end-point $v$ of the conditioned diffusion will be fixed throughout. To emphasize the dependence of the transition density on the first two arguments and to shorten notation, we will often write
\[p(s,x) = p(s,x; T, v).\]
Motivated by the guiding term in the drift of $\pin X$ (see \eqref{bstar}), we further introduce the notations
\label{sec:main result}
\begin{center}\rara{1.2}
\begin{tabular}{|l|}
\hline
$R(s,x) = \log p(s,x), \quad
	 r(s,x) = \DD R(s,x), \quad  H(s,x) = -\DD^2 R(s,x).$\\
\hline
\end{tabular}
\end{center}
Here $\DD$ acts on $x$. Similarly the functions $\tilde R$, $\tilde r$ and $\tilde H$ 
are defined by starting with the transition densities $\tilde p$  in the place of $p$. 

The following proposition deals with the laws of the processes $X$, $X^\circ$ and $X^\star$ on the interval $[0,t]$ for 
$t < T$ (strict inequality is essential). Equivalence of these laws is clear from Girsanov's theorem. The 
proposition gives expressions for the corresponding Radon-Nikodym derivatives, which are 
derived using Kolmogorov's backward equation. The proof of this result can be found 
in Section \ref{sec:proof_theorem}.

\begin{prop}\label{prop:likeli}
Assume for all $x, y\in \RR^d$ and $t\in [0,T)$ 
\begin{align}
\label{eq:bound-rtilde}
\|\tilde r(t, x)\| \lesssim 1+ \frac{\norm{x-v}}{T-t}, \qquad 
\|\tilde r(t, y) - \tilde r(t,x)\| \lesssim \frac{\norm{y-x}}{T-t}.
\end{align}
Define the process $\psi$  by 
\begin{equation}
\label{eq:defpsi} \psi(t)=\exp\left(\int_0^t G(s,X^\circ_s) \dd s \right), \quad t < T, 
\end{equation}
where
\begin{align*} G(s,x) &= (b(s,x) - \tilde b(s,x))^\T \tilde r(s,x) \\ & \qquad -  \frac12 \trace\left(\left[a(s, x) - \tilde a(s, x)\right] \left[\tilde H(s,x)-\tilde{r}(s,x)\tilde{r}(s,x)^\T\right]\right).
\end{align*}
Then  for $t \in [0,T)$ the laws $\P_{t}$, $\P^\circ_{t}$ and $\P^\star_{ t}$ 
are equivalent and we have
\begin{align}\nonumber
\frac{\dd \P_{t}}{\dd \P^\circ_{t}}({X^\circ}) & =
 \frac{\tilde p(0,u; T, v)}{\tilde p(t, X^\circ_{t}; T, v)} \: \psi(t),\\[1ex]
\label{likelib}
\frac{\dd \P^\star_{t}}{\dd \P^\circ_{t}}({X^\circ}) & =
 \frac{\tilde p(0,u; T, v)}{p(0,u; T, v)}\frac{p(t, X^\circ_{t}; T, v)}{\tilde p(t, X^\circ_{t}; T, v)} \: \psi(t).
\end{align}
\end{prop}

\bigskip

 Proposition \ref{prop:likeli} is not of much use for simulating diffusion bridges unless its statements can be shown to hold in the limit $t\uparrow T$ as well. One would like to argue that in fact we have equivalence 
 of measures on the whole interval $[0,T]$ and that 
\begin{equation}
\label{eq:goal}	\frac{ \dd\P^\star_{T}}{\dd\P^\circ_{T}}({X^\circ}) =
 \frac{\tilde p(0,u; T, v)}{p(0,u; T, v)}  \: \psi(T). 
\end{equation}
As $\psi(T)$ does not depend on $p$, samples from $X^\circ$ can then be used as proposals for 
$X^\star$ in a Metropolis-Hastings sampler, for instance. Numerical evaluation of $\psi(T)$ is somewhat simplified by the fact  that no  stochastic integral appears in its expression. 
To establish \eqref{eq:goal} we need to put  appropriate conditions on the processes  $X$ and $\tilde{X}$
that allow us to control the behaviour of the bridge processes $X^*$ and $X^\circ$ near time $T$.

\begin{ass}
\label{ass:tildeX}
For the auxiliary process $\tilde{X}$ we assume the following: 
\begin{itemize}
\item[(i)] For all bounded, continuous functions $f\colon [0,T]\times\RR^d \to \RR$ the transition densities  $\tilde p$ of  $\tilde X$ satisfy
\begin{equation}\label{weakcont}
\lim_{t \uparrow T} \int f(t, x) \tilde p(t, x; T, v) \dd x = f(T, v).
\end{equation}
\item[(ii)] For all $x, y\in \RR^d$ and $t\in [0,T)$, the functions $\tilde{r}$ and $\tilde{H}$ satisfy 
\begin{align*}\|\tilde r(t, x)\|& \lesssim 1+ \norm{x-v}(T-t)^{-1} \\ 
\|\tilde r(t, x) - \tilde r(t,y)\| &\lesssim {\norm{y-x}}{(T-t)^{-1}}\\
 \|\tilde{H}(t, x)\| &\lesssim    (T-t)^{-1} +   \norm{x-v}(T-t)^{-1}.\end{align*}
\item[(iii)]  There exist  constants $\tilde{\Lambda},  \tilde{C} > 0$ such that for $0 < s < T$, 
$$ \tilde{p}(s,x; T, v) \le 
\tilde{C} (T-s)^{-d/2} \exp\left(-\tilde{\Lambda}\frac{\|v-x\|^2}{T-s}\right)$$ uniformly in $x$. 
\end{itemize}
\end{ass}

Roughly speaking, Assumption \ref{ass:tildeX} requires that the process $\tilde X$, 
which we choose ourselves, is  
sufficiently nicely behaved near time $T$. 

\begin{ass}
\label{ass:relaronson-for-p}
For $M>1$ and $u\ge 0$ define $g_M(u)=\max(1/M, 1-Mu)$. 
There exist constants $\Lambda,  C>0$,  $M > 1$ and a function $\mu_t(s,x)\colon \{s,t\colon 0\le s \le t \le T\}\times \RR^d\to\RR^d$ 
with $\|\mu_t(s,x) - x\|<M(t-s)\|x\|$ and $\|\mu_t(s,x)\|^2 \ge g_M(t-s)  \|x\|^2$, so that for all $s<t\le T$ and $x, y\in\RR^d$,
 \[
 p(s,x; t, y) \le  C(t-s)^{-d/2} \exp\left(-\Lambda\frac{\|y - \mu_t(s,x) \|^2}{t-s}\right).
 \]
\end{ass}
Assumption \ref{ass:relaronson-for-p} refers to the generally unknown transition densities of $X$. In case the drift of $X$ is bounded, assumption \ref{ass:relaronson-for-p} is implied by the stronger Aronson's inequality (cf.\ \cite{Aronson}). However, assumption \ref{ass:relaronson-for-p} also holds for example for linear processes which in general have unbounded drift.

\begin{ass}\label{ass:Xcirc}
There exist an $\eps \in (0,1/6)$  and  an a.s.\ finite random variable $M$ such that for all $t\in [0,T]$,
it a.s.\ holds that 
\[ \|X^\circ_t-v\| \le M (T-t)^{1/2-\eps}. \]
\end{ass}

This third assumption requires that the proposal process $X^\circ$ does not only 
converge to $v$ as $t \uparrow T$, as it obviously should, but that it does so 
at an appropriate speed. 
A requirement of this kind can not be essentially avoided, since 
in general two bridges can only be equivalent if 
they are pulled to the endpoint with the same force. 
Theorem \ref{thm:linearproc} below asserts that this assumption holds in case 
$\tilde X$ is a linear process, provided its diffusion coefficient 
coincides with that of the process $X$ at the final time $T$.

We can now state the main  results of the paper.

\begin{thm}\label{thm:criterion}
Suppose that Assumptions  \ref{ass:tildeX}, \ref{ass:relaronson-for-p} and \ref{ass:Xcirc} hold
and that $\tilde{a}(T,v)=a(T,v)$. Then the laws of the bridges $X^\star$ and $X^\circ$ 
are equivalent on $[0,T]$ and \eqref{eq:goal} holds, with $\psi$ as in Proposition 
\ref{prop:likeli}.
\end{thm}

We complement this general theorem with a result that asserts, 
as already mentioned,  
that Assumptions \ref{ass:tildeX} and \ref{ass:Xcirc} 
hold for a class of processes $\tilde X$ given by linear SDEs.

\begin{thm}\label{thm:linearproc}
Assume $\tilde X$ is a linear process with dynamics governed by the stochastic differential equation
\begin{align}\label{linsdehomog}
		\dd  \tilde X_t = \tilde B(t) \tilde X_t\dd t  + \tilde \beta(t)  \dd t +  \tilde \sigma(t)  \dd  W_t,
\end{align}
for non-random matrix and vector functions $\tilde B$, $\tilde\beta$ and $\tilde\si$. 
\begin{itemize}
\item[(i)] If 
$\tilde B$ and $\tilde\beta$ are  continuously differentiable on $[0,T]$,  
$\tilde\si$ is Lipschitz on $[0,T]$ and 
there exists an $\eta>0$ such that for all $s\in [0,T]$ and all  $y\in \RR^d$,
\[	
y^\T \tilde{a}(s) y \ge \eta \|y\|^2,
\]
then $\tilde{X}$ satisfies Assumption  \ref{ass:tildeX}. 
\item[(ii)] Suppose moreover that 
 $\tilde a(T)=a(T,v)$, that 
there exists an $\eps >0$ such that for all $s \in [0,T]$, $x\in \RR^d$ and $y \in \RR^d$ 
\begin{equation}
\label{eq:ellipt}
 	 y^\T a(s,x) y \ge \eps  \|y\|^2,
\end{equation}
and that 
$b$ is of the form
$b(s,x) = B(s,x)x + \beta(s,x)$, 
where $B$ is a bounded  matrix-valued function and $\beta$ is a bounded vector-valued function.
Then there exists an a.s.\ finite random variable $M$ such that, a.s.,  
\[ 
\|X^\circ_t-v\| \le M \sqrt{(T-t)\log \log \left(\frac1{T-t}+e\right)}
\]
for all $t \in [0,T]$. In particular, Assumption \ref{ass:Xcirc} holds for any $\eps > 0$. 
\end{itemize}
\end{thm}

The proofs of Theorems \ref{thm:criterion} and \ref{thm:linearproc} can be found
in Sections \ref{sec:proof-thm-criterion}--\ref{sec:proof_linearproc}.

\begin{rem}
Extending absolute continuity of  $\pin{X}$ and $X^\circ$ on $[0,T-\eps]$ ($\eps>0$) to  absolute continuity on $[0,T]$ is a subtle issue.
This  can already be seen from a very simple example in the one-dimensional case. Suppose $d=d'=1$,  $v=0$, $b\equiv 0$ and $\si(t,x)\equiv 1$.  That is, $\pin X$ is the law of a  Brownian bridge from $0$ at time $0$ to $0$ at time $T$ satisfying the stochastic differential equation
\[ \dd \pin{X}_t= -\frac{\pin{X}_t}{T-t} \dd t +\dd W_t. \]
Suppose we take $\tilde{X}_t=\tilde{\sigma} \dd W_t$, so that $X^\circ$
satisfies  the stochastic differential equation
\[ \dd X^\circ_t= -\frac1{\tilde{\sigma}^2} \frac{X^\circ_t}{T-t} \dd t +  \dd W_t. \]
It is a trivial fact that $X^\circ$ and $X^\star$ are absolutely continuous on $[0,T]$ if $\tilde\sigma=1$ (this also follows from theorem \ref{thm:linearproc}). It is natural to wonder whether this condition is also necessary. The answer to this question is yes, as we now argue.  Lemma 6.5 in  \cite{HidaHitsuda} gives a general result on absolute continuity of Gaussian measures. From this result it follows that $X^\circ$ and $X^\star$ are absolutely continuous on $[0,T]$ if and only if for the symmetrized Kullback-Leibler divergences
\[ d_t =\expec{ \log \frac{\dd \P^\star_t}{\dd \P^\circ_t}({X^\star}) }  +\expec{\log \frac{\dd \P^\circ_t}{\dd \P^\star_t}({X^\circ})}  \]
it holds that $\sup_{t\in [0,T)} d_t <\infty$. 
We consider the second term. Denoting  $\alpha = 1/\tilde \sigma^2$, Girsanov's theorem gives 
\begin{align*}
  \log \frac{\dd \P^\circ_t}{\dd \P^\star_t}({X^\circ})  =
   \int_0^t (1-\alpha) \frac{X^\circ_s}{T-s}  \dd W_s
    + \frac12  \int_0^t  ( \alpha-1)^2 \left(\frac{X^\circ_s}{T-s}\right)^2 \dd s
\end{align*}

By It\=o's formula $\frac{X^\circ_t}{T-t} = (1-\alpha) \int_0^t \frac{ X^\circ_s}{(T-s)^{2}} \dd s
+   \int_0^t  \frac{-1}{T-s}  \dd W_s.$

This is a  linear equation with solution 
\[\textstyle\frac{X^\circ_t}{T-t}
 = -\left(T-t\right)^{-1+\alpha} \int_0^t \left(T-s\right)^{- \alpha} \dd W_s,
\] hence 
\[\textstyle \expec{\left(\frac{X^\circ_t}{T-t}\right)^2} =  
\left(T-t\right)^{-2+2\alpha} \int_0^t \left(T-s\right)^{-2\alpha} \dd s\]

For $t < T$,
$\int_0^t \expec{\left(\frac{X^\circ_s}{T-s} \right)^2} \dd s <\infty$, so 
$\expec{\int_0^t \frac{X^\circ_s}{T-s}  \dd W_s} = 0$. 
Therefore
\begin{equation*}
\expec{\log \frac{\dd \P^\circ_t}{\dd \P^\star_t}({X^\circ}) } =   \frac12 ( \alpha-1)^2   \int_0^t \left(T-s\right)^{-2+2\alpha} \int_0^s \left(T-\tau\right)^{-2\alpha} \dd \tau \dd s.
\end{equation*}
Unless, $\alpha=1$, this diverges for  $t \uparrow T$. We conclude that the laws of $X^\star$ and $X^\circ$ are singular if  $\alpha \ne 1$.
\end{rem}

\begin{rem}
For implementation purposes integrals in likelihood ratios and solutions to stochastic differential equations  need to be approximated on a finite grid. This is a subtle numerical issue as the drift of our proposal bridge  has a singularity near its endpoint. In a forthcoming work \cite{vdMeulenSchauer} we show how this problem can be  dealt with.  The main idea in there is the  introduction of a time-change and space-scaling of the proposal process that allows for numerically accurate discretisation and evaluation of the likelihood.
\end{rem}

\section{Proof of Proposition \ref{prop:likeli}}\label{sec:proof_theorem}

We first note that by equation \eqref{eq:bound-rtilde}, $\tilde{r}$ is Lipschitz in its second argument on  $[0, t]$  
and satisfies a linear growth condition. Hence, a unique strong solution of the SDE for $X^\circ$ exists.

By Girsanov's theorem (see e.g.\ \cite{LiptserShiryayevI})
the laws of the processes $X$ and $X^\circ$ on $[0,t]$ are equivalent and the corresponding 
Radon-Nikodym derivative is given by 
\[
\frac{\dd\P_{t}}{\dd\P^\circ_{t}}({X^\circ})= 
\exp\Big(\int_0^{t} \beta_s^\T  \dd W_s  - \frac12 \int_0^{t} \|\beta_s\|^2 \dd s\Big), 
\]
where $W$ is a Brownian motion under $\P^\circ_{t}$  and $\beta_s = \beta(s, X^\circ_s)$ solves
$\sigma(s, X^\circ_s) \beta(s, X^\circ_s)  =  b(s, X^\circ_s)-b^\circ(s, X^\circ_s).$
(Here we  lightened notation by writing $\beta_s$ instead of $\beta(s, X^\circ_s)$. In the 
remainder of the proof we follow the same convention and apply it to  other processes as well.) 
Observe that by definition of $\tilde r$ and $b^\circ$ we have
$\beta_s =  -\sigma^\T_s \tilde{r}_s$ and $\|\beta_s\|^2 = \tilde{r}^\T_s a_s \tilde{r}_s$, hence 
\[
\frac{\dd\P_{t}}{\dd\P^\circ_{t}}({X^\circ})= 
\exp\Big(-\int_0^{t} \tilde{r}_s^\T \sigma_s  \dd W_s  - \frac12 \int_0^{t} \tilde{r}^\T_s a_s \tilde{r}_s \dd s\Big). 
\]
Denote the infinitesimal operator of $X^\circ$ by $\Ell^\circ$. By definition of $X^\circ$ 
and $\tilde R$ we have $\Ell^\circ \tilde R = \Ell \tilde R + \tilde{r}^\T a \tilde{r}$.
By   It\=o's formula, it follows that 
\[
\tilde R_t - \tilde R_0 = \int_0^t \Big(\frac{\partial}{\partial s}\tilde R + \Ell \tilde R\Big)\,\dd s + 
\int_0^t \tilde{r}^\T_s a_s \tilde{r}_s\,\dd s 
+ \int_0^t  \tilde{r}_s^\T\sigma_s\,\dd W_s.
\]
Combined with what we found above we get $\frac{\dd\P_{t}}{\dd\P^\circ_{t}}({X^\circ})= \e^{-(\tilde R_t - \tilde R_0)} \e^{\int_0^t G_s\dd s}, 
$ where $
\left( \frac{\partial}{\partial s} \tilde R+ \Ell \tilde R\right) +\frac12\tilde{r}^\T a  \tilde{r}$. By Lemma \ref{lem:time} ahead the first  term between brackets  on the right-hand-side of this display equals
$ \Ell \tilde{R}-\tilde\Ell \tilde{R} -\frac12 \tilde{r}^\T\tilde{a} \tilde{r}$.
Substituting this in the expression for $G$ gives
\begin{align*} 
	G =
  (b-\tilde b)^\T \tilde{r}  -\frac12 \trace\left( (a-\tilde{a}) \tilde{H}\right)+\frac12 \tilde{r}^\T (a-\tilde{a})\tilde{r} ,
\end{align*}
which is as given in the statement of the theorem.  Since 
$ -(\tilde R_t - \tilde R_0) = \log {  \tilde{p}(0,u)}/{\tilde{p}(t,X^\circ_t)}$,
we arrive at the first assertion of the proposition.  

To prove the second assertion, 
let   $0=t_0<t_1 < t_2 < \dots < t_N< t < T$ and define $x_0=u$. If  $g$ is  a bounded function on $\RR^{(N+1)}$, then standard calculations show $\expec{g(X^\star_{t_1}, \ldots, X^\star_{t_N}, X^\star_{t}) \frac{1}{p(t,X^\star_t)}} = \expec{g(X_{t_1}, \ldots, X_{t_N}, X_{t}) \frac{1}{p(0,u)} }
$, using the abbreviation $p(t,x) = p(t, x; T, v)$.
Since the grid and $g$ are arbitrary, this proves that for $t < T$, 
\begin{equation}\label{eq:ster}
\frac{\dd\P^\star_{t}}{\dd\P_t}(X) = \frac{p(t, X_t; T, v)}{p(0,u; T, v)}.
\end{equation}
Combined with the first statement of the proposition, this yields the second one.

\begin{lemma}\label{lem:time}
  $\tilde R$ satisfies the equation
\[
\frac{\partial}{\partial s} \tilde R + \tilde\Ell \tilde R  =-  \frac12  \tilde r^\T \tilde a \tilde  r .
\]
\end{lemma}
\begin{proof}
First note that
\begin{equation}\label{derivatives}	\DD^2_{ij} \tilde R(s,x)= \frac{\DD^2_{ij} \tilde p(s,x)}{\tilde p(s,x)}
- \left(\DD_i \tilde R(s,x)\right) \left(\DD_j  \tilde R(s,x)\right).
\end{equation}
Next,  Kolmogorov's backward equation is given by
\[	 \frac{\partial}{\partial s} \tilde p(s,x)+ \left(\tilde\Ell  \tilde p\right)(s,x)=0. \]
Dividing both sides by $\tilde p(s,x)$ and using (\ref{infgen}) we obtain
\begin{equation*}	\frac{\partial}{\partial s} \tilde R(s,x)=-
 \sum_{i=1}^{d}  \tilde b_i(s,x) \DD_i \tilde R(s,x) - \frac12\sum_{i,j=1}^{d} \tilde a_{ij}(s,x)\frac{\DD^2_{ij}  \tilde p(s,x) }{\tilde p(s,x)}
\end{equation*}
Now substitute (\ref{derivatives}) for the second term on the right-hand-side  and re-order terms to get the result.
\end{proof}

\section{Proof of Theorem \ref{thm:criterion}}
\label{sec:proof-thm-criterion}
Auxiliary lemmas used in the proof are gathered in Subsection \ref{subsec:res_thm-criterion} ahead.
As before we use the notation $p(s,x) = p(s, x; T, v)$ and similar for $\tilde p$. Moreover, we define
$\bar p = \frac{\tilde p(0,u)}{p(0, u)}.$
The main part of the proof consists in proving that $\bar p \psi(T)$ is 
indeed a Radon-Nikodym derivative, i.e.\ that it has expectation $1$. 
For $\eps \in (0,1/6)$ as in Assumption \ref{ass:Xcirc},  
  $m\in \NN$  and  a stochastic process $Z=(Z_t,\, t \in [0,T])$, define 
$$\sigma_m(Z) = T \wedge \inf_{t \in [0,T]} \{\abs{Z_t - v} \ge m(T-t)^{1/2-\eps}\}.$$ We suppress the dependence on $\eps$ in the notation. 
We write
$  \sigma_m = \sigma_m(X),
\sigma_m^\star = \sigma_m(X^\star)$, and $\sigma_m^\circ = \sigma_m(X^\circ)$.
Note that $\si_m^\circ \uparrow T$ holds in probability,  by Assumption \ref{ass:Xcirc}.

By Proposition \ref{prop:likeli}, for any $t<T$ and bounded, $\mathcal F_t$-measurable $f$, we have
\begin{equation}\label{likeliexp}
 \expec{ f(X^\star) \frac{ \tilde  p(t, X^\star_t)}{p(t, X^\star_t)}} =
\expec{ f(X^\circ)  \bar{p} \: \psi(t)}.
\end{equation}
By Corollary \ref{cor:psibound} in Subsection \ref{subsec:res_thm-criterion}, for each $m\in\NN$,  $\sup_{0\le t\le T}\psi(t)$ is
uniformly  bounded  on the event $\{T=\si_m^\circ\}$. Hence, by dominated convergence, 
\begin{align*}
 \expec{ \bar{p}\, \psi(T) \ind{T=\si_m^\circ}} 
=  \lim_{t\uparrow T} \expec{ \bar{p}\, \psi(t) \ind{t\le\si_m^\circ}}
 \le  \lim_{t\uparrow T} \expec{  \bar{p}\,  \psi(t) }
  =  \lim_{t\uparrow T} \expec{ \frac{\tilde{p}(t,\pin{X}_t)}{p(t,\pin{X}_t)}}=1.
\end{align*}
Here the final two  equalities follow from equation \eqref{likeliexp} and  Lemma \ref{lem:tildepoverp}, 
respectively. Taking the limit $m\to \infty$ we obtain 
$\expec{\bar{p} \,\psi(T) } \le 1$, 
by monotone convergence. 
For the reverse inequality note that by similar arguments as just used we obtain
$$\expec{\bar p\, \psi(T)} \ge \expec{\bar p\, \psi(T)  \ind{T = \sigma_m^\circ}} = \lim_{t \uparrow T} \expec{ \bar p\, \psi(t)  \ind{t \le \sigma^\circ_m} }= \lim_{t \uparrow T} \expec{ \frac{\tilde p(t, X^\star_t)}{p(t, X^\star_t)} \ind{t \le \sigma^\star_m}}. 
$$
By Lemma \ref{lem:expec_greater_1}, the right-hand-side of the preceding display tends to $1$ as $m\to \infty$. 
We conclude that $\bar{p}\, \expec{\psi(T)}=1$. 

To complete the proof we note that by 
 equation (\ref{likeliexp})  and Lemma \ref{lem:tildepoverp} we have  $\bar p\, \expec{\psi(t)} \to 1$ as $t\uparrow T$. 
 In view of the preceding and Scheff\'e's Lemma this implies that 
$\psi(t) \to \psi(T)$ in $L^1$-sense as $t \uparrow T$.  
Hence for $s < T$ and a bounded,  $\mathcal{F}_s$-measurable functional $g$, 
\[
\expec{g(X^\circ) \bar p \psi(T)} = \lim_{t \uparrow T} \expec{ g(X^\circ)
\frac{\tilde p(t, X^\circ_{t})}{p(t, X^\circ_{t})}\left(\bar p \frac{p(t, X^\circ_{t})}{\tilde p(t, X^\circ_{t})}  \psi(t)\right)}.
\]
Proposition \ref{prop:likeli}  implies that for $t > s$, the expectation on the right  equals  
\[
\expec{g(X^\star)
\frac{\tilde p(t, X^\star_{t})}{p(t, X^\star_{t})}}.
\]
By Lemma \ref{lem:tildepoverp} this converges to $\EE\, g (X^\star)$ as $t \uparrow T$ and we find that 
$\EE\, g(X^\circ) \bar p \psi(T) = \EE\, g (X^\star)$. Since $s < t$ and $g$ are arbitrary, this completes the proof.

\subsection{Auxiliary results used  in the proof of Theorem \ref{thm:criterion}}
\label{subsec:res_thm-criterion}
\begin{lemma}\label{lem:ftsxbound}
Suppose Assumptions \ref{ass:tildeX}(iii) and  \ref{ass:relaronson-for-p} apply.
For 
\begin{equation}\label{eq:ftsx}
f_t(s,x) =  \int  p(s,x;t,z)\, \tilde p(t,z; T, v)  \dd  z \qquad 0\le s< t <T,\: x\in \RR^d,
\end{equation}
there exist positive constants $c$ and $\lambda$ such that
\[ f_t(s,x) \le c  (T-s)^{-d/2} \exp\left(- \lambda \frac{\|v - x\|^2}{T-s}\right). \]
\end{lemma}

\begin{proof}
Let $C, \tilde{C}, \Lambda$ and $\tilde\Lambda$ be the constant appearing in assumptions \ref{ass:tildeX}(iii) and \ref{ass:relaronson-for-p}. Define  $\bar\Lambda=\min(\Lambda, \tilde\Lambda)/2$. Denote by $\phi(z;\mu,\Sigma)$ the $N(\mu,\Sigma)$-density, evaluated at $z$. 
 Then there exists a $\bar{C}>0$ such that
\begin{align*}
f_t(s,x)  &\le \bar{C} \int \phi(z; \mu_t(s,x), \bar\Lambda^{-1}(t-s)\I_d)\phi(v-z; 0, \bar\Lambda^{-1}(T-t)\I_d)\dd z
\\ & = \bar{C} \phi(v; \mu_t(s,x), \bar\Lambda^{-1}(T-s)\I_d).
\end{align*}
Using the second assumed bound on $\mu_t(s,x)$ and the fact that $g_M(t-s)\ge 1/M$ we get
\[ \|v-\mu_t(s,x)\|^2 \ge M^{-1}\|v-x\|^2 + (1-g_M(t-s)) \|v\|^2 -2 v^\T(\mu_t(s,x)-g_M(t-s) x). \]
By Cauchy-Schwarz, the triangle inequality and the first assumed inequality we find 
\[ \left| v^\T(\mu_t(s,x)-g_M(t-s) x)\right| \le \|v\| \|x\| \left( M(t-s) + 1-g_M(t-s)\right). \]
We conclude that 
\begin{align*}
\frac{\|v-\mu_t(s,x)\|^2}{T-s} \ge & \frac1{M}\frac{\|v-x\|^2}{T-s} + \frac{1-g_M(t-s)}{T-s} \|v\|^2  \\ & - 2 \left( \frac{M(t-s)}{T-s} + \frac{1-g_M(t-s)}{T-s}\right) \|v\| \|x\| . 
\end{align*}
By definition of $g_M$, the multiplicative terms appearing in front of $\|v\|^2$ and $\|v\| \|x\|$ are both bounded. 
As there exist constants $D_1>0$ and $D_2 \in \RR$ such that the  third term on the right-hand-side can be lower bounded by $D_1 \|v-x\|^2 + D_2$ the result follows. 
\end{proof}

The following lemma is similar to Lemma 7 in \cite{DelyonHu}.

\begin{lemma}\label{lem:tildepoverp}
Suppose Assumptions \ref{ass:tildeX}(i), \ref{ass:tildeX}(iii) and  \ref{ass:relaronson-for-p} apply. 
If $0<t_1 < t_2 < \dots < t_N< t < T$ and  $g \in C_b(\RR^{Nd})$, then
$$\lim_{t \uparrow T}  \E\left[ \,g(X^\star_{t_1}, \dots, X^\star_{t_N}) \frac{\tilde p(t, X^\star_{t})}{p(t, X^\star_{t})}\right]  = \expec{g(X^\star_{t_1}, \dots, X^\star_{t_N})}.$$
\end{lemma}

\begin{lemma}
\label{lem:calcbounds}
Assume
\begin{enumerate}
\item $b(s,x)$, $\tilde{b}(s,x)$, $a(s,x)$ and $\tilde{a}(s,x)$ are locally Lipschitz in $s$ and  globally Lipschitz in $x$;
\item $\tilde{a}(T,v)=a(T,v)$.
\end{enumerate}
Then for all $x$ and for all $s\in [0,T)$,  
\begin{equation}\label{diffb}
 \|b(s,x)-\tilde{b}(s,x)\| \lesssim 1 + \|x-v\|  
 \end{equation}
and 
\begin{equation}\label{diffa} 
 \|a(s,x)-\tilde{a}(s,x)\|_F \lesssim (T-s) + \|x-v\|.
\end{equation}
If in addition $\tilde{r}$ and $\tilde{H}$ satisfy the bounds
\begin{align*}
&	\|\tilde r(s, x)\| \lesssim 1+\norm{x-v}(T-s)^{-1} \\ &  \|\tilde{H}(s, x)\|_F \lesssim    (T-s)^{-1} +   \norm{x-v}(T-s)^{-1},
\end{align*}
then
\[ \left| G(s,x)\right| \lesssim 1+ (T-s) + \|x-v\| +\frac{\|x-v\|}{T-s}+\frac{\|x-v\|^2}{T-s}+\frac{\|x-v\|^3}{(T-s)^2}. \]
\end{lemma}
\begin{proof}
Since $|\mbox{tr}\,(AB)| \le \|A\|_F \|B\|_F$ and $\|A B\|_F \le \|A\|_F \|B\|_F$ for compatible matrices $A$ and $B$, we have
\begin{align}
\label{eq:boundD}
 \left|G(s,x)\right| & \le \|b(s,x)-\tilde{b}(s,x)\| \|\tilde{r}(s,x)\| +\nonumber \\ 
& \qquad \|a(s,x)-\tilde{a}(s,x)\|_F \left( \|\tilde{H}(s,x)\|_F + \|\tilde{r}(s,x)\|^2\right).
\end{align}
Bounding $\|b(s,x)-\tilde{b}(s,x)\|$ proceeds by using the assumed Lipschitz properties for $b$ and $\tilde{b}$. We have
\begin{align*}
\|b(s,x)-\tilde{b}(s,x)\| &  \le \|b(s,x)-b(s,v)\| + \|b(s,v)-\tilde{b}(s,v)\| 
+ \|\tilde{b}(s,v)-\tilde{b}(s,x)\| \\ & \le L_b \|x-v\| + \|b(s,v)-\tilde{b}(s,v)\| + L_{\tilde{b}} \|v-x\|,
\end{align*}
where $L_b$ and $L_{\tilde{b}}$ denote Lipschitz constants. Since $b(\cdot, v)$ and $\tilde{b}(\cdot, v)$ are continuous on $[0,T]$, we have $\|b(s,v)-\tilde{b}(s,v)\| \lesssim 1$. This inequality together with preceding display gives \eqref{diffb}. 

Bounding $\|a(s,x)-\tilde{a}(s,x)\|_F$ proceeds by using the assumed Lipschitz properties for $a$ and $\tilde{a}$ together with $\tilde{a}(T,v)=a(T,v)$. We have
\begin{align*}
 \|a(s,x)-\tilde{a}(s,x)\|_F & \le   \|a(s,x)-a(T,x)\|_F+\|a(T,x)-a(T,v)\|_F+\|a(T,v)-\tilde{a}(T,v)\|_F \notag\\&\qquad+ \|\tilde{a}(T,v)-\tilde{a}(s,v)\|_F+ \|\tilde{a}(s,v)-\tilde{a}(s,x)\|_F \notag\\& \lesssim (T-s) + \|x-v\|.
\end{align*}
The final result follows upon plugging in the derived estimates for $\|b(s,x)-\tilde{b}(s,x)\|$ and $\|a(s,x)-\tilde{a}(s,x)\|_F$ into equation \eqref{eq:boundD} and subsequently using the bounds on $\tilde{r}$ and $\tilde{H}$ from the assumptions of the lemma. 
\end{proof}


\begin{cor}
\label{cor:psibound}
Under the conditions of Lemma \ref{lem:calcbounds}, for all  $\eps \in (0,1/6)$ there is a positive constant $K$ (not depending on $m$) such that for all $t\in [0,T)$ 
\[	\psi(t) \ind{t\le \si_m^\circ} \le \exp\left(K m^3 \right).\]
\end{cor}
\begin{proof}
On the event $\{ t\le \si_m^\circ \}$ we have 
\[ \|X^\circ_s-v\| \le m (T-s)^{1/2-\eps} \qquad \text{for all $s\in [0,t]$.}   \]
Together with the result of Lemma \ref{lem:calcbounds}, this implies that there is a constant $C>0$ (that does not depend on $m$) such that for all $s\in [0,t]$
\begin{align*} \left| G(s,X^\circ_s)\right| &\le C \left( 1+m(T-s)^{1/2-\eps} + m(T-s)^{-1/2-\eps} + m^2 (T-s)^{-2\eps} + m^3 (T-s)^{-1/2-3\eps}\right) \\ & \le Cm^3  \left( 1+(T-s)^{1/2-\eps} + (T-s)^{-1/2-3\eps}\right).
\end{align*}
Hence,
\[ \psi(t) \ind{t\le \si_m^\circ} \le \exp\left(Cm^3 \int_0^T \left( 1+(T-s)^{1/2-\eps} + (T-s)^{-1/2-3\eps}\right) \dd s \right) \le \exp\left(K m^3\right), \]
for some constant $K$.
\end{proof}

\begin{lemma}\label{lem:expec_greater_1}
Suppose Assumptions \ref{ass:tildeX}(i),  \ref{ass:tildeX}(iii) and \ref{ass:relaronson-for-p} apply. 
 Then 
\[	
\lim_{m\to \infty} \lim_{t \uparrow T} \expec{ \frac{\tilde p(t, X^\star_t)}{p(t, X^\star_t)} \ind{t \le  \sigma^\star_m}}=1.  
\]
\end{lemma}
\begin{proof}
First, 
 \[ \expec{ \frac{\tilde p(t, X^\star_t)}{p(t, X^\star_t)} \ind{t \le \sigma^\star_m}} =  \expec{ \frac{\tilde p(t, X^\star_t)}{p(t, X^\star_t)}}  - \expec{ \frac{\tilde p(t, X^\star_t)}{p(t, X^\star_t)} \ind{t > \sigma_m^\star}}.\]
Hence,  by Lemma \ref{lem:tildepoverp},  it suffices to prove that the second term tends to $0$.
For $t < T$
\begin{align*}
  p(0,u) & \,\expec{ \frac{\tilde p(t, X^\star_t)}{p(t, X^\star_t)}\ind{t >\sigma_m^\star}}
 = \expec{\tilde p(t, X _t)  \ind{t >\sigma_m} } \\
& = \expec{ \expec{ \tilde p(t, X _t) \ind{t >\sigma_m} \mid  \mathcal F_{\sigma_m}}}
= \expec{ \ind{t >\sigma_m} \expec{ \tilde p(t, X _t)  \mid  \mathcal F_{\sigma_m}}} \\ &=
\expec{ \ind{t >\sigma_m}    \int p(\sigma_m,X_{\sigma_m}; t,z) \tilde p(t ,z) \dd z }
 = \expec{\ind{t >\sigma_m} f_t(\si_m, X_{\si_m})},
\end{align*}
where $f_t$ is defined in equation \eqref{eq:ftsx}. 
Here we used \eqref{eq:ster} and  the strong Markov property.
By Lemma  \ref{lem:ftsxbound},
\[	\expec{f_t(\si_m, X_{\si_m})} \lesssim  \expec{
(T-\si_m)^{-d/2} \exp\left(-\lambda \frac{\|v-X_{\sigma_m}\|^2}{T-\sigma_m}\right) }. \]
Since $\|v-X_{\sigma_m}\| = m (T-\si_m)^{1/2-\eps}$, the right-hand-side 
 can be bounded by a constant times $\expec{ (T-\si_m)^{-d/2} \exp\left(-\lambda m^2(T-\si_m)^{-2\eps}\right)}.$ Note that this expression does not depend on $t$. The proof is concluded by taking the limit $m\to \infty$. Trivially, $T-\si_m \in [0,T]$, so that the preceding display can be bounded by 
\[ C\sup_{\tau \in [0,\infty)} \tau^{-d/2} \exp\left(-\la m^2 \tau^{-2\eps}\right) \le C\left(\frac{d}{4\la m^2  \e \eps}\right)^{\frac{{d}}{{4\eps}}}. \]
This tends to $0$ as $m\to \infty$. 
\end{proof}

\section{Proof of Theorem \ref{thm:linearproc}(i)}
\label{sec:linproc}

It is well known (see for instance \cite{LiptserShiryayevI})  that the linear process 
$\tilde X$ is a Gaussian process that can be described in terms of the fundamental $d\times d$ matrix $\Phi(t)$, which satisfies
\[ \Phi(t) = \I + \int_0^t \tilde B(\tau) \Phi(\tau) \dd \tau. \]
We define  $\Phi(t,s)=\Phi(t)\Phi(s)^{-1}$, 
\begin{equation}\label{mugen}
	\mu_t(s,x) = \Phi(t,s) x + \int_s^t \Phi(t,\tau) \tilde \beta(\tau) \dd \tau
\end{equation}
and
\begin{equation}\label{Kgen}
 K_t(s)    = \int_{s}^t \Phi(t,\tau) a(\tau) \Phi(t,\tau)^\T \dd \tau. 
\end{equation}
To simplify notation, we use the convention that whenever the subscript $t$ is missing, it has the value of the end time $T$.  So we write  $\mu(s,x)=\mu_T(s,x)$ and $K(s)=K_T(s)$.
The Gaussian transition densities of  the process $\tilde X$ can be explicitly expressed 
in terms of the objects just defined. In particular we have
\begin{equation}\label{R-linear} \tilde R(s,x)
= -\frac{d}{2}\log(2\pi)-\frac12 \log |K(s)| -\frac12(v-\mu(s,x))^\T K(s)^{-1} (v-\mu(s,x))  .
\end{equation}
This will allow us to derive explicit expressions for all the functions involved in Assumption \ref{ass:tildeX}.

For future purposes, we state a number of properties of $\Phi(t,s)$, which are well known in  literature on linear differential equations (proofs can be found for example in Sections 2.1.1 up till 2.1.3 in \cite{Chicone}).
\begin{itemize}
\item $\Phi(t,s)\Phi(s,\tau)=\Phi(t,\tau)$, $\Phi(t,s)^{-1}=\Phi(s,t)$ and $\frac{\partial \Phi}{\partial s} (t,s)=-\Phi(t,s) B(s)$. 
\item There is a constant $C\ge 0$ such that for all $s, t \in [0,T]$, $\|\Phi(t,s)\| \le C$ (this is a consequence of Gronwall's lemma). 
\item $|\Phi(t,s)| = \exp\left( \int_s^t \trace(\tilde B(u)) \dd u \right)$ (Liouville's formula). 
\item If $\tilde B(t)\equiv \tilde B$ does not depend on $t$, $\Phi(t,s)=\exp(\tilde B(t-s))=\sum_{k=0}^\infty{\frac{1}{k!}}\tilde B^k (t-s)^k$. 
\end{itemize}
By Theorem 1.3 in \cite{Chicone}, we have that the mappings $(t,s,x) \mapsto \mu_t(s,x)$ and $(t,s) \mapsto \Phi_t(s)$ are continuously differentiable.

The following lemma provides the explicit expressions for the functions $\tilde r$ and $\tilde  H$.

\begin{lemma}\label{lemma:relation_H-r}
For  $s\in [0,T)$ and $x\in \RR^d$
 \begin{equation*}
\tilde r(s,x)=\DD \tilde R(s,x)=   \Phi(T,s)^\T K(s)^{-1} (v-\mu(s,x))
\end{equation*}
and
\begin{align}\label{eq:H-linear}
\tilde H(s, x) = \tilde H(s) &=-\DD \tilde r(s,x)=  \Phi(T,s)^\T K(s)^{-1} \Phi(T,s)\nonumber \\ &=\left(\int_{s}^T \Phi(s,\tau) \tilde a(\tau) \Phi(s,\tau)^{\T} \dd \tau\right)^{-1}.
\end{align}
Moreover, we have the relation $\tilde r(s,x) = \tilde H(s) (v(s)-x)$
where
\begin{equation}\label{def_v}v(s) =  \Phi(s,T) v - \int_s^T \Phi(s,\tau) \tilde \beta(\tau) \dd \tau.
\end{equation}
\end{lemma}

\begin{proof}
We use the conventions and rules on differentiations outlined in Section \ref{sec:notation}. Since $K(s)$ is symmetric
\begin{align*}
\tilde r(s,x) &=  - \DD(v-\mu(s,x))  K(s)^{-1} (v-\mu(s,x))
\\& =\Phi(T,s)^\T K(s)^{-1} (v-\mu(s,x)) ,
\end{align*}
where we used $\DD \mu(s,x)=\Phi(s)^\T$.

By equation \eqref{mugen},
\begin{equation}\label{eq:v-mu} 
v-\mu(s,x)=v-\Phi(T,s) x -\int_s^T \Phi(T,\tau)  \tilde\beta(\tau) \dd \tau. 
\end{equation}
The expression for $\tilde H$ now  follows from
\begin{align*}
	\tilde H(s)  & = -\DD(\Phi(T,s)^\T K(s)^{-1} (v-\mu(s,x)))\\
 &=  \DD(\Phi(T,s)^\T K(s)^{-1} \Phi(T,s)x) = \Phi(T,s)^\T K(s)^{-1} \Phi(T,s) ,
\end{align*}
where the second equality follows from equation \eqref{eq:v-mu}. 

The final statement follows upon noting that
\begin{align*} \tilde r(s,x)&= \Phi(T,s)^\T K(s)^{-1} \Phi(T,s) \Phi(s,T) (v-\mu(s,x)) \\&= \tilde H(s)\Phi(s,T) (v-\mu(s,x))=\tilde H(s)( v(s)-x).
\end{align*}
The last equality follows by multiplying equation \eqref{eq:v-mu} from the left with $\Phi(s,T)$. 
\end{proof}

In the following three  subsections we use the explicit computations 
of the preceding lemma to verify  Assumption \ref{ass:tildeX}, 
in order to complete the proof statement (i) of Theorem \ref{thm:linearproc}.

\subsection{Assumption \ref{ass:tildeX}(i)}

\begin{lemma}\label{lem:weakconv}
If  $f\colon [0,T]\times\RR^d  \to \RR$ is bounded  and continuous then
 $$\lim_{t \to T} \int f(t, z) \tilde p(t, z; T, v)\dd z = f(T,v).$$
\end{lemma}
\begin{proof}
The log of the transition density of a linear process is given in equation (\ref{R-linear}). Using $v$ as defined in \eqref{def_v} and the expression for  $\mu$ as given in \eqref{mugen}, we get
\[ \mu(t,x)=\Phi(T,t)\left( x + \Phi(t,T) v -v(t)\right)= \Phi(T,t)(x-v(t)) +v. \]
This gives
\[	A(t,x):=(v-\mu(t,x))^\T K(t)^{-1} (v-\mu(t,x))=(\Phi(T,t)(x-v(t))^\T  K(t)^{-1} \Phi(T,t) (x-v(t))\]
It follows that we can write
\[ \int f(t, x) \tilde p(t, x; T, v)\dd z  = \int \frac{f(t,x)}{\sqrt{|K(t)|}} (2\pi)^{-d/2} \exp\left( -\frac12 A(t,x) \right)\dd x. \]
Upon  substituting $z=\Phi(T,t)(x-v(t))$ this equals
\[	\int f(t,\Phi(t,T) z +v(t)) (2\pi)^{-d/2} \frac1{\sqrt{|K(t)|}} \exp\left(-\frac12 z^\T K(t)^{-1} z\right) |\Phi(t,T)| \dd z.\]
We can rewrite this expression as $\expec{W_t}$ where
\begin{equation*} W_t= |\Phi(t,T)| f(t,\Phi(t,T) Z_t+v(t)) .
\end{equation*}
and  $Z_t$ denotes  a random vector with $N(0,K(t))$-distribution. As $t\uparrow T$, $Z_t$ converges weakly to a Dirac mass at zero. As $\Phi(t,T)$ converges to the identity matrix and $v(t) \to v$, we get that $\Phi(t,T) Z_t +v(t)$ converges weakly to $v$. By the continuous mapping theorem and continuity of $f$, $W_t$ converges weakly to $f(T,v)$. Since the limit is degenerate, this statement holds for convergence in probability as well. By boundedness of  $f$, we get $\expec{W_t} \to f(T,v)$.
\end{proof}

\subsection{Assumption \ref{ass:tildeX}(ii)}

\begin{lemma}
\label{lem:Hbound}
There exists a positive constant $C$ such that for all $s \in [0,T)$ and $x, y \in \RR^d$
\begin{align}
&	(T-s) \|\tilde H(s)\| \le C,
\\& \|\tilde r(s,x)\| \le C \left( 1 + \frac{\|v-x\|}{T-s}\right),
 \label{eq:bound-rsx}
\\&\|\tilde r(s,y) - \tilde r(s,x)\| \le C\frac{\|y-x\|}{T-s}
\\&  \frac{\|v-x\|}{T-s} \le C \left(1+ \|\tilde r(s,x)\|\right). 
\label{eq:vxr}
\end{align}
\end{lemma}
\begin{proof}
In the proof, we  use the relations proved in Lemma \ref{lemma:relation_H-r}. From this lemma it follows that
\[	
\tilde H(s)^{-1}=\int_{s}^T \Phi(s,\tau) a(\tau) \Phi(s,\tau)^{T} \dd \tau. 
\]
Since $\Phi(s,\tau)$ is uniformly bounded and $\tau \mapsto \tilde a(\tau)$ is continuous, it easily follows that
$y^{\T} \tilde H(s)^{-1} y \le \tilde{c} (T-s) \|y\|^2$  for all $y \in \RR^d$.
By uniform ellipticity of $\tilde a$, there exists a constant $c_1>0$ such that  for all  $y \in \RR^d$ 
\[ y^\T \Phi(s,\tau) \tilde a(\tau) \Phi(s,\tau)^{\T} y \ge c_1 y^\T \Phi(s,\tau) \Phi(s,\tau)^\T y. \]
Secondly, there exists a constant $c_2 > 0$ such that $y'\Phi(s,\tau) \Phi(s,\tau)^\T y \ge c_2 \|y\|^2$ uniformly in $s, \tau \in [0, T]$.
To see this, suppose this second claim is false. 
Then for each $n \in \NN$ there are $s_n, \tau_n \in [0,T]$, $y_n \in \RR^d\setminus\{0\}$ such that 
$ \| \Phi(s_n,\tau_n)^{\T} y_n \|^2 \le \frac1n \|y_n\|^2$, or
letting $z_n = y_n/\|y_n\|$,
\[ \| \Phi(s_n,\tau_n)^{\T} z_n \|^2 \le \frac1n.\]
By compactness of the set $[0,T]^2\times \{z \in \RR^d, \|z\| = 1\}$ and by continuity of $\Phi$, there exists a convergent subsequence  $s_{n_i}, \tau_{n_i}, z_{n_i} \to s^*, \tau^*, z^*$, such that,
$\|\Phi(s^*,\tau^*)^{\T} z^* \|^2 = 0$ with  $z^* \ne 0$. This contradicts Liouville's formula.

Integrating over $\tau \in [s,T]$ gives
\begin{equation}
\label{eq:Hs-inv} 
	y^\T \tilde H(s)^{-1}y \ge c (T-s) \|y\|^2, 
\end{equation}
where $c=c_1 c_2$. 
Hence, we have proved that 
\[	c \|y\|^2 \le y^\T ((T-s) \tilde H(s))^{-1}y \le \tilde{c} \|y\|^2. \]
Since $\tilde H$ is symmetric, this says that the eigenvalues of the matrix $((T-s)\tilde H(s))^{-1}$ are contained in the interval $[c, \tilde{c}]$. This implies that the eigenvalues of $(T-s)\tilde H(s)$ are in $[1/\tilde{c}, 1/c]$.   
 Since the operator norm of a  positive definite matrix is bounded by its largest eigenvalue, it follows that 
 $(T-s)\|\tilde H(s)\|\le 1/c$.  

To prove the second inequality, note that 
\begin{align*}
 \tilde r(s,x)&= \tilde H(s) (v(s)-x)=\tilde H(s) \left[v(s)-v(T) + v-x\right]\\
& = (T-s) \tilde H(s) \left[-\frac{v(T)-v(s)}{T-s} + \frac{v-x}{T-s}. \right]
\end{align*}
Now 
\[ v(T)-v(s)=\left(\Phi(T,T) - \Phi(s,T)\right) v + \int_s^T \Phi(s,\tau) \tilde \beta(\tau) \dd \tau. \]
As $s\mapsto \Phi(s,T)$ is continuously differentiable, we have
\[ \|v(T)-v(s)\| \le C_1 (T-s) + \int_s^T \|\Phi(s,\tau)\| \|\tilde \beta(\tau)\| \dd \tau \le C_2 (T-s). \]
Hence,
\[	
  \|\tilde r(s,x)\| \le (T-s) \|\tilde H(s)\| \left( C_2 + \frac{\|v-x\|}{T-s} \right)
\]
which yields \eqref{eq:bound-rsx}. 
Also,
\[\| \tilde r(s,x) - \tilde r(s,y)\| =
\|\tilde H(s)(y-x)\| \lesssim \frac{\|y - x\|}{T - s}. \]

For obtaining the fourth inequality of the lemma, 
\[ \tilde H(s) (v-x) = \tilde r(s,x) + \tilde H(s)(v(T)-v(s)). \]
Upon multiplying both sides by $((T-s)\tilde H(s))^{-1}$ this gives
\[ \frac{\|v-x\|}{T-s} \le \|((T-s)\tilde H(s))^{-1}\|  \|r(s,x)\| + \frac{\|v(T)-v(s)\|}{T-s}. \]
Substitution of the derived bounds on $\tilde H(s)^{-1}$ and $v(T)-v(s)$ completes the proof.
\end{proof}

\subsection{Assumption \ref{ass:tildeX}(iii)}

\begin{lemma}
\label{ass:aronson-linearproc}
There exist positive  constants $C$ and $\Lambda$ such that for all $s \in [0,T)$
\begin{equation}
\label{eq:pbound-linearproc}
   \tilde p(s,x; T, v) \le C (T-s)^{-d/2} \exp\left(-\Lambda\frac{\|v-x\|^2}{T-s}\right).
\end{equation}
\end{lemma}
\begin{proof}
Using the relations from Lemma \ref{lemma:relation_H-r} together with equation \eqref{R-linear}, some straightforward calculations yield
\[  
\tilde R(s,x)
= -\frac{d}{2}\log(2\pi) -\frac12 \log|K(s)| -\frac12 \tilde r(s,x)^\T \tilde H(s)^{-1} \tilde r(s,x)  .
\]
By \eqref{eq:Hs-inv}, there exists a positive constant $c_1 > 0$ such that 
\[	\tilde r(s,x)^\T \tilde H(s)^{-1} \tilde r(s,x)  \ge c_1(T-s) \|\tilde r(s,x)\|^2. \]
By equation \eqref{eq:vxr} the right-hand-side is lower bounded by
\[ c_1 \left\{\max\left(\frac{\|x-v\|}{\sqrt{T-s}} - c_2 \sqrt{T-s},0\right) \right\}^2 \]
for some positive constant $c_2$. Now if $a\ge 0$ and $b \in [0,c_2]$, then there exist $c_3, c_4 >0$ such that $\left(\max(a-b,0)\right)^2 \ge c_3 a^2 -c_4$ (this is best seen by drawing a picture). Applying this with $a=\|v-x\|/\sqrt{T-s}$ and $b=c_2 \sqrt{T-s}$ gives
\[	\tilde r(s,x)^\T \tilde H(s)^{-1} \tilde r(s,x)  \ge c_1 \left(c_3 \frac{\|v-x\|^2}{T-s}-c_4\right). \]
This yields the exponential bound in \eqref{eq:pbound-linearproc}.

Since $\tilde H(s)^{-1}=\Phi(s,T) K(s) \Phi(s,T)^{T}$ we have
$|K(s)| = \frac{|\Phi(T,s)|^2}{|\tilde H(s)|}.$ Multiplying both sides by $(T-s)^{-d}$ gives 
 \[ (T-s)^{-d} |K(s)| = \frac{|\Phi(T,s)|^2}{|(T-s)\tilde H(s)|}. \] 
Since the eigenvalues of $(T-s)\tilde H(s)$ are bounded by $1/c$ uniformly over $s\in [0,T]$ (see Lemma \ref{lem:Hbound}) and the determinant of a symmetric matrix equals the product of its eigenvalues, we get
 \[ (T-s)^{-d} |K(s)| \ge  |\Phi(T,s)|^2 c^d =c^d \exp\left(2\int_s^T \trace(\tilde B(u)) \dd u \right). \] 
by Liouville's formula. Now it follows that the right-hand-side of the preceding display is bounded away from zero uniformly over $s\in [0,T]$. 
\end{proof}

\section{Proof of Theorem \ref{thm:linearproc}(ii)}\label{sec:proof_linearproc}
Auxiliary results used in the proof are gathered in Subsection \ref{sec:proof_ui_lemmas} ahead.

By \eqref{eq:vxr} in Lemma \ref{lem:Hbound} we have 
$\|x - v\| \lesssim (T-t) (1+\|\tilde r(t,x)\|)$. Therefore we focus on bounding $\|\tilde r (t, x)\|$. 
Define $w$ to be the positive definite square root of $a(T,v)$. 
Then it follows from our assumptions that $\|w\| < \infty$ and $\|w^{-1}\|< \infty$, 
hence we can equivalently derive a bound for  
$\tilde{Z}(s,x)=w\, \tilde{r}(s,x).$ We do this in two steps. First we obtain a preliminary bound by 
writing an SDE for $\tilde Z$ and bounding the terms in the equation. Next we 
strengthen the bound using a Gronwall-type inequality. 

By Lemma \ref{lem:sde-rtilde}, $\tilde{Z}$ satisfies the stochastic differential equation
\begin{equation}\label{sde_Z}
\dd \tilde{Z}(s,X^\circ_s)=-w\tilde{H}(s) \si(s,X^\circ_s) \dd W_s + \Upsilon(s, X^\circ_s) \dd s+\Delta(s,X^\circ_s) \tilde{Z}(s,X^\circ_s)  \dd s,
\end{equation}
where
\begin{align}\label{eq:Delta}
 \Delta(s,X^\circ_s)&=  w \left( \tilde{H}(s) \left(\tilde{a}(s)-a(s,X^\circ_s)\right)- \tilde{B}(s)\right)w^{-1} \\
\Upsilon(s,X^\circ_s)&=w \tilde{H}(s)\left(\tilde{b}(s,X^\circ_s)-b(s,X^\circ_s)\right).
\end{align}
Define $\tilde{J}(s)=w \tilde{H}(s)w$.
For $\Delta$ we have the decomposition $\Delta=\Delta_1+\Delta_2 + \Delta_3$, with
\begin{align}
  \Delta_1(s,X^\circ_s) &= \frac{1}{T-s} \left(\I -w^{-1} a(s,X^\circ_s) w^{-1}\right) \label{eq:def:Delta1} \\ 
  \Delta_2(s,X^\circ_s) &= \left(\tilde{J}(s)- \frac{1}{T-s}\right) \left(\I -w^{-1} a(s,X^\circ_s) w^{-1}\right)\nonumber  \\ 
\Delta_3(s) &= w \left[\tilde{H}(s) \left( \tilde{a}(s)-\tilde{a}(T)\right) -\tilde{B}(s)\right] w^{-1}.\nonumber
\end{align}
To see this, we calculate
$
	\Delta_1(s,X^\circ_s) + \Delta_2(s,X^\circ_s) =  \tilde{J}(s) \left(\I -w^{-1} a(s,X^\circ_s) w^{-1}\right)
$ and
\[
 \Delta(s,X^\circ_s)-\Delta_1(s,X^\circ_s) -\Delta_2(s,X^\circ_s)  = w\left[\tilde{H}(s) \tilde{a}(s) - \tilde{B}(s)\right] w^{-1} - \tilde{J}(s). \]
Upon substituting
$\tilde{J}(s)= w\tilde{H}(s)a(T,v) w^{-1}=w\tilde{H}(s) \tilde{a}(T)w^{-1}$
into this display we end up with exactly $\Delta_3(s)$.

For $\Upsilon$ we have a decomposition $\Upsilon = \Upsilon_1 \tilde{Z} + \Upsilon_2$
with
\begin{align*}
\Upsilon_1(s,X^\circ_s) &=  w \tilde H(s) (B(s,X^\circ_s) - \tilde B(s)) \tilde H^{-1}(s) w^{-1}  \\
\Upsilon_2(s,X^\circ_s) &=  w \tilde H(s)[\tilde \beta(s)-\beta(s)   -(B(s,X^\circ_s)-\tilde B(s)) v(s)].
\end{align*}
Here, ${v}(s)$ is as defined in \eqref{def_v}. 
To prove the decomposition, first note that $\Upsilon$, $\Upsilon_1$ and $\Upsilon_2$ share the factor $w\tilde{H}(s)$. Therefore, it suffices to prove that 
\begin{multline}
\label{eq:decomp_Ups}
  \tilde{b}(s,x)-b(s,x)- (B(s,x) - \tilde B(s)) \tilde H^{-1}(s) w^{-1} \tilde{Z}(s,x)  \\  = \tilde{\beta}(s)-\beta(s,x)   -(B(s,x)-\tilde B(s)) v(s). 
\end{multline}
By Lemma \ref{lemma:relation_H-r}, $\tilde{Z}(s,x)=w\tilde{r}(s,x)=w\tilde{H}(s)\left(\tilde{v}(s)-x\right)$. Upon substituting this into the left-hand-side of the preceding display we obtain
\[ \left(\tilde{B}(s)-B(s,x)\right) x + \tilde{\beta}(s)-\beta(s,x) - \left(B(s,x)-\tilde{B}(s)\right)\left(\tilde{v}(s)-x\right) , \]
which is easily seen to be equal to the right-hand-side of \eqref{eq:decomp_Ups}.
Thus, \eqref{sde_Z} can be written as
 \begin{multline}\label{sde_Z_transformed}
 \dd \tilde{Z}(s,X^\circ_s) =-w\tilde{H}(s)  \si(s,X^\circ_s)  \dd W_s \\ + \left[\Delta_1(s,X^\circ_s) + \Delta_2(s,X^\circ_s)+ \Delta_3(s) + \Upsilon_1(s,X^\circ_s)\right] \tilde{Z}(s,X^\circ_s)  \dd s + \Upsilon_2(s,X^\circ_s) \dd s.
 \end{multline}
Next, we derive bounds on $\Delta_1$, $\Delta_2$,  $\Delta_3$, $\Upsilon_1$ and $\Upsilon_2$. 
\begin{itemize}
\item By Lemma \ref{lem:Delta1bound} it follows that there is a $\eps_0 \in (0,1/2)$  such that 
\[ y^\T \Delta_1(s,X^\circ_s)y \le \frac{1-\eps_0}{T-s}\norm{y}^2 \quad \text{for all} \quad s\in [0,T)\quad \text{and} \quad y \in \RR^d.\] 
\item By Lemma \ref{lem:Delta2bound}, $\|\tilde{J}(s)-\I/(T-s)\|$ is bounded for  $s\in [0,T]$.  As  $\sigma$ is bounded, this implies 
$\Delta_2$ can be bounded by deterministic constant $C_1>0$.
\item 
For $\Delta_3$, we  employ the Lipschitz property of $\tilde{a}$ to deduce that there is a deterministic constant $C_2>0$ such that 
\[ \|\Delta_3(s)\| \le (T-s) \|\tilde{H}(s)\|\left\| \frac{\tilde{a}(s)-\tilde{a}(T)}{T-s}\right\| + \|\tilde{B}(s)\| \le C_2 . \]
\item 
Since $(s,x)\mapsto B(s,x)$ is assumed to be bounded, there exists a deterministic constant $C_3>0$ such that
\[  \|\Upsilon_1(s,X^\circ_s)\|\le    \|B(s,X^\circ_s) - \tilde B(s)\| \le C_3. \]
\item   
Similarly, using that  $s\mapsto \tilde{v}(s)$ is bounded on $[0,T]$,  we have the existence of a deterministic constant $C_4$ such that 
\begin{multline*}  (T-s) \|\Upsilon_2(s)\| =\\ \|w\| (T-s)\left\|\tilde{H}(s)\right\|\left[\|\tilde{\beta}(s)\|+\|\beta(s,X^\circ_s)\|   +\|B(s,X^\circ_s)-\tilde B(s)\| \|v(s)\|\right] \le C_4.
\end{multline*}
\end{itemize}
Now we set $A(s,x) = \Delta_1(s,x) +  \Delta_2(s,x) + \Delta_3(s) + \Upsilon_1(s,x)$
and let $\Psi(s)$ be the principal fundamental matrix at 0 for the corresponding random homogeneous linear system
\begin{equation}
\label{defPsi}
\dd \Psi(s)= A(s, X_s^\circ)  \Psi(s)\dd s,\qquad \Psi(0)=\I.
 \end{equation}
Since $s\mapsto A(s,X_s^\circ)$ is continuous for each realization $X^\circ$, $\Psi(s)$ exists uniquely (\cite{Chicone}, Theorem 2.4). 
 Using the just derived bounds, for all $y \in \RR^d$
\[	y^{\T} A(s,X^\circ_s) y \le \frac{1-\eps_0}{T-s} \|y\|^2 + C_1+C_2+C_3. \]
By  Lemma \ref{lem:SidedPsiBound}, this implies existence of a positive constant $C$ such that 
\[ \| \Psi(t) \Psi(s)^{-1}\| \le C \left(\frac{T-s}{T-t}\right)^{1-\eps_0}, \qquad 0\le s\le t<T.\]
By  Lemma \ref{lem:linsol},  for $s<T$ we can represent $\tilde{Z}$ as 
 \begin{equation}\label{represent-Z2}	\tilde{Z}(s,X^\circ_s)= \Psi(s) \tilde{Z}(0,u) +  \Psi(s) \int_0^s  \Psi(h)^{-1} \Upsilon_2(h)  \dd h -  M_s, \end{equation}
where
\begin{equation}\label{defhatM}	 M_s = \Psi(s) \int_0^s \Psi(h)^{-1}  w \tilde{H}(h) \sigma(h, X^\circ_h)  \dd W_h. \end{equation}
Bounding $\|\tilde{Z}(s,X^\circ)\|$ can be done by bounding the norm of each term on the right-hand-side of equation \eqref{represent-Z2}.

The norm of the first term can be bounded by 
$\|\tilde{Z}(0,u)\| \|\Psi(s)\| \lesssim (T-s)^{\eps_0-1}$. 
The norm of the second one can be bounded by 
\[ 
 \int_0^s  \left(\frac{T-h}{T-s}\right)^{1-\eps_0} \frac{1}{T-h} \|\Upsilon_2(h)(T-h)\| \dd h \lesssim (T-s)^{\eps_0-1}.
 \]
For the third term,  it follows from  Lemma \ref{lem:integ}, 
applied  with $U(s,h) =   w \tilde{H}(h)\si(h,X^\circ_h)$,  that   there is an 
a.s.\ finite  random variable $\overline{M}$ such that for all $s<T$
$\left\| M_s \right\|  \le \overline{M} (T-s)^{\eps_0-1}.$ 
Therefore, there exists a random variable $\overline{M}'$ such that  
\begin{equation}\label{eq: rho}	
\|\tilde{Z}(s,X^\circ_s)\| \le \overline{M}' (T-s)^{\eps_0-1}. 
\end{equation}

We finish the proof by showing that the bound obtained can be improved upon.
We go back to equation \eqref{sde_Z} and consider the various terms. 
By inequality \eqref{diffb} and the inequalities of Lemma \ref{lem:Hbound} we can bound
\begin{equation*} \|\Upsilon(s,x)\| \lesssim   \| \tilde{H}(s)\| \left(1 + \|x-v\|\right)   \lesssim (T-s)^{-1} + \frac{\|v-x\|}{T-s} \lesssim 1+ (T-s)^{-1} + \|\tilde{Z}(s,x)\|. 
\end{equation*}
Similarly, using inequality \eqref{diffa}
\begin{equation*} \|\Delta(s,x) \| \lesssim 1 + \frac{\|v-x\|}{T-s}   \lesssim 1+ \|\tilde{Z}(s,x)\|.
\end{equation*}
The quadratic variation $\qv{L}$ of the martingale part 
$L_t = \int_0^tw\tilde{H}(s) \si(s,X^\circ_s) \dd W_s$
is given by $\qv{L}_t = \int_0^t w\tilde H(s) a(s, X^\circ_s)\tilde H(s) w\dd s.$
Hence, by the boundedness of $\|\tilde{H}(s)(T-s)\|$ we have
\[
\|\qv{L}_t\| \lesssim \int_0^t \frac1{(T-s)^2}\dd s = \frac{1}{T-t} - \frac1T \le \frac1{T-t}.
\]
By the Dambis-Dubins-Schwarz time-change theorem and the law of the iterated logarithm 
of Brownian motion, it follows that there exists an a.s.\ finite random variable $N$ such that 
$\|L_t\| \le Nf(t)$ for all $t < T$, where 
\[
f(t) = \sqrt{\frac1{T-t}\log \log \left(\frac1{T-t} + e\right)}.
\]
Taking the norm on the left- and right-hand-side of equation \eqref{sde_Z}, applying the derived bounds 
and using that $\int_0^t(T-s)^{-1} \dd s \lesssim \sqrt{1/({T-t})}$ we get with $\rho(s) = \|  \tilde{Z}(s,X^\circ_s)\|$ that 
$\rho(t) \le N f(t) +  C\int_0^t \left( \rho(s) + \rho^2(s)\right) \dd s$, $t < T$ for some positive constant $C$.
The bound \eqref{eq: rho} derived above implies that $\rho$ is 
integrable on $[0,T]$. 
The proof of assertion (ii) of Theorem \ref{thm:linearproc} is now completed 
by applying Lemma \ref{lem:improve}.




\subsection{Auxiliary results used  in the proof of Theorem \ref{thm:linearproc}(ii)}\label{sec:proof_ui_lemmas}
\begin{lemma}\label{lem:V(s)lipschitz}
Define $V(s)= w^{-1} \tilde H(s)^{-1} w^{-1}$ and $V'(s)=\frac{\partial}{\partial s}V(s).$
It holds that $s\mapsto V'(s)$ is Lipschitz on $[0,T]$ and $V'(s) \to -\I$ as $s\uparrow T$. 
\end{lemma}
\begin{proof}
By equation \eqref{eq:H-linear}
\[ 
\Phi(T,s) \tilde H(s)^{-1} \Phi(T,s)^\T= \int_s^T \Phi(T,\tau) \tilde a(\tau) \Phi(T,\tau)^\T \dd \tau. \]
Taking the derivative with respect to $s$ on both sides and reordering terms gives
\[ \frac{\partial}{\partial s} \tilde H(s)^{-1} = -\tilde a(s) +\tilde B(s) \tilde H(s)^{-1} + \tilde H(s)^{-1} \tilde B(s)^\T, \]
and hence $V'(s)= w^{-1} \left(-\tilde a(s) +\tilde B(s) \tilde H(s)^{-1} + \tilde H(s)^{-1} \tilde B(s)^\T\right) w^{-1}.$ Since $\|\Phi(s,\tau)\| \le C$ for all $s, \tau \in [0,T]$, it follows that $s\mapsto V'(s)$ is Lipschitz on $[0,T]$. Furthermore,  $V'(s) \to -w^{-1}a(T) w^{-1}=-\I$, as $s\uparrow T$.  
\end{proof}


\begin{lemma}\label{lem:sde-rtilde}
We have 
\begin{align*}\dd \tilde{r}(s,X^\circ_s)&=-\tilde{H}(s) \si(s,X^\circ_s) \dd W_s\\ &  \qquad +    \tilde{H}(s)\left(\tilde{b}(s,X^\circ_s)-b(s,X^\circ_s)\right)\dd s \\& \qquad+ \left(\tilde{H}(s) \left(\tilde{a}(s)-a(s,X^\circ_s)\right)- \tilde{B} \right) \tilde{r}(s,X^\circ_s) \dd s, \end{align*}
where $\tilde{B}=D\tilde{b}$.
\end{lemma} 
\begin{proof}
In the proof, we will omit dependence on $s$ and $X^\circ_s$ in the notation.
By It\=o's formula
\begin{equation}\label{tilder}  \dd \tilde{r}= \frac{\partial}{\partial s}\tilde{r} \dd s - \tilde{H} d X^\circ. \end{equation}
For handling the second term we  plug-in the expression for $X^\circ$ from its defining stochastic differential equation. This gives
\begin{equation}\label{tilder2}	 \tilde{H} \dd X^\circ = \tilde{H}b  \dd s +\tilde{H} a \tilde{r} \dd s + \tilde{H} \sigma \dd W. \end{equation}
 For the first term, we  compute the derivative of $\tilde{r}(s,x)$ with respect to $s$. For this, we note that  by  
 Lemma \ref{lem:time} $\frac{\partial }{\partial s} \tilde{R} = -\tilde\Ell \tilde{R}-\frac12 \tilde{r}^\T\tilde{a}\tilde{r}$,
with
$\tilde{\Ell}\tilde{R} = \tilde{b}^\T \tilde{r}-\frac12 \trace\left(\tilde{a}\tilde{H}\right).$
Next, we take $D$ on both sides of this equation.
Since we assume $\tilde{R}(s,x)$ is differentiable in $(s,x)$ we have $D\left((\partial/\partial s) \tilde R\right)= (\partial/\partial s) \tilde{r}$. Further,
$D\left(\tilde{\Ell} \tilde{R}\right) = \tilde{B} \tilde{r} - \tilde{H} \tilde{b}$ 
and $D\left(\frac12 \tilde{r}^\T\tilde{a}\tilde{r}\right) = - \tilde{H} \tilde{a} \tilde{r}.$
Therefore, $\frac{\partial}{\partial s} \tilde{r}= -\tilde{B}\tilde{r}+\tilde{H}\tilde{b}+\tilde{H}\tilde{a}\tilde{r}.$ Plugging this expression together with (\ref{tilder2}) into equation \eqref{tilder} gives the result.
\end{proof}

\begin{lemma}\label{lem:Delta1bound} 
 There  exists an $\eps_0 \in (0,1/2)$   such that for $0\le s < T$, $x, y \in \RR^d$
\begin{equation*}\label{deltapos}
y^\T\Delta_1(s,x)y \le \left(\frac{1- \eps_0}{T-s} \right) \norm{y}^2, 
\end{equation*}
with $\Delta_1$ as defined in \eqref{eq:def:Delta1}.
\end{lemma}
\begin{proof}
Let  $y\in \RR^d$. By \eqref{eq:ellipt} there is $\eps > 0$ such that
\[
y^\T\Delta_1(s,x)y  =  y^\T  \left(\frac1{T-s} \right) (\I - w^{-1} a(s,x)w^{-1})  y \le 
\left(\frac1{T-s} \right)\left( y^\T y-\eps y^\T \tilde{a}(T)^{-1} y \right).
\]
Since $\tilde{a}(T)=a(T,v)$ is positive definite, its inverse is positive definite as well. Hence, there exists a $\eps'>0$ such that $y^\T \tilde{a}(T)^{-1} y \ge \eps' \|y\|^2$. This gives  $y^\T\Delta_1(s,x)y \le \frac{1-\eps\eps'}{T-s} \|y\|^2.$ Let $\eps_0=\eps\eps'$. We can take $\eps$ sufficiently small such that $\eps_0 \in (0,1/2)$. 
\end{proof}

\begin{lemma}\label{lem:Delta2bound}
Let $\tilde{J}(s)=w \tilde{H}(s) w$.  
There exists a $C > 0$ such that 
\[\left\|\tilde{J}(s)- \frac{1}{T-s}\I\right\| < C \quad\text{for all $s < T$.}\]

\end{lemma}
\begin{proof}
We have
\begin{equation}\label{eq:tildeJtermproof}	
\left\|\tilde{J}(s)- \frac{1}{T-s}\I\right\|  \le \frac1{T-s} \left\| \tilde{J}(s)\right\| \left\| (T-s) \I - \tilde{J}^{-1}(s)\right\|. 
\end{equation}
Let $\tilde{V}(s)=\tilde{J}(s)^{-1}$ and $\tilde{V}'(s)=  \frac{\partial}{\partial s}\tilde{V}(s)$. Since $\tilde{V}(T)=0$ and  $\tilde{V}'(T)=-\I$ (see Lemma \ref{lem:V(s)lipschitz}) we can write 
\[ (T-s) \I -\tilde{V}(s)=-\int_s^T \tilde{V}'(T) + \int_s^T \tilde{V}'(h) \dd  h. \]
By Lemma \ref{lem:V(s)lipschitz}, $s\mapsto \tilde{V}'(s)$ is Lipschitz on $[0,T]$ and therefore
\[ \left\|  (T-s) \I -\tilde{V}(s) \right\| \lesssim  \int_s^T (T-h) \dd h =(T-s)^2/2. \]
Substituting the derived bound into  \eqref{eq:tildeJtermproof} gives
\[	\left\|\tilde{J}(s)- \frac{1}{T-s}\I\right\|  \lesssim (T-s) \left\| \tilde{J}(s)\right\| \lesssim (T-s) \left\| \tilde{H}(s)\right\| \lesssim 1. \]
The last inequality follows from Lemma \ref{lem:Hbound}. 
\end{proof}

\begin{lemma}
\label{lem:SidedPsiBound}
Let  $\Psi(t)$ be the principal fundamental matrix at 0 for the random homogeneous linear system
\begin{equation}
\dd \Psi(s)=  A(s)  \Psi(s)\dd s,\qquad \Psi(0)=I.
 \end{equation}
Suppose that the matrix function $A(s)$ is of the form $A(s) = A_1(s) + A_2(s)$, where  both $A_1$ and $A_2$ are continuous on $[0, T)$. Assume  $A_2$ is bounded and $A_1$ is such that there are  $\eps_0  \in (0,1/2)$ and  $C_1>0$  that for all $s  \in [0, T)$ and vectors $y$
\[  y^\T A_1(s) y \le \left(\frac{1-\eps_0}{T-s} + C_1\right) \norm{y}^2. \]
Then there is a $C > 0$ such that for all $0\le s\le t<T$ 
\[ \| \Psi(t) \Psi(s)^{-1}\| \le C \left(\frac{T-s}{T-t}\right)^{1-\eps_0}. \]
\end{lemma}
\begin{proof}

For $z \in \RR^d$, let $Z(t) = \Psi(t) z$, so $\dd Z(t) = (A_1(t) + A_2(t))Z(t) \dd t$. Let $\norm{A_2(t)} \le  C_2$ (say).
Integrating  $\dd [Z(u)^\T Z(u)] = \dd [Z(u)^\T] Z(u) + Z(u)^\T [\dd Z(u)] =
Z(u)^\T (A_1 + A_2 + A_1^\T + A_2^\T) Z(u) \dd u$ over $[s,t]$ yields
\begin{align*}
Z(t)^\T Z(t) &= Z(s)^\T Z(s) + \int_s^t  Z(h)^\T (A_1(h) + A_1(h)^\T) Z(h) \dd h\\&\quad +   \int_s^t Z(h)^\T ( A_2(h)+ A_2(h)^\T)Z(h) \dd h \\ &
\le Z(s)^\T Z(s) +  \int_s^t 2 \left(\frac{1-\eps_0}{T-h} + C_1 + C_2\right) Z(h)^\T Z(h) \dd h. 
\end{align*}
From Gronwall's lemma,
\[\norm{Z_t}^2 \le \norm{Z_s}^2   \exp\left(2\int_s^t  \frac{1-\eps_0}{T-u} \dd u + 2(t-s)(C_1 + C_2) \right).
\]
Let $z = \Psi(s)^{-1} x$. For any $x$ with $\norm{x} \le 1$ this implies \[ \|\Psi(t)\Psi(s)^{-1} x \|  \le\norm{\Psi(s)\Psi(s)^{-1}x}\left(\frac{T-s}{T-t}\right)^{1-\eps_0 }\e^{(t-s)(C_1+C_2)}\]
or $\|\Psi(t)\Psi(s)^{-1} \|  \le \e^{T(C_1+C_2)} \left (\frac{T-s}{T-t}\right)^{1-\eps_0}.$
\end{proof}


\begin{lemma}\label{lem:linsol}
Suppose $Y$ is a strong solution of the  stochastic differential equation $\dd Y_t = \alpha_t \dd W_t+ (\beta_t + \gamma_t Y_t) \dd t,$ where $\alpha_t=\alpha(t,Y_t)$, $\beta_t=\beta(t,Y_t)$ and $\ga_t=\gamma(t,Y_t)$. Let  $\Psi$ be the matrix solution to
$\dd \Psi(t)=\ga_t \Psi(t) \dd t$, $\Psi(0)=\I$ and define the process $Y'$ by
\[ Y'_t = \Psi(t) \left[ Y_0 + \int_0^t \Psi(h)^{-1} \beta_h \dd h + \int_0^t \Psi^{-1} \alpha_h d W_h\right].\]
If $\sup_{s\le \tau} \|\ga_s\|<\infty$,  then $Y$ and $Y'$ are indistinguishable on  $[0,\tau]$.
\end{lemma}
\begin{proof}
By computing $\int_0^t \ga_s Y'(s) \dd s$ and using the (stochastic) Fubini theorem it is easy to verify that $Y'$ satisfies the stochastic differential equation
\[ \dd Y'_t= \alpha_t \dd W_t+ (\beta_t + \gamma_t Y'_t) \dd t.\]
This implies $Y'_s-Y_s= \int_0^t \ga_s(Y'_s-Y_s) \dd s$ and thus
\[ \sup_{s\le t} \|Y'_s-Y_s\| \le \max_{s\le t} \|\ga_s\| \int_0^t \sup_{h\le s} \|Y'_h-Y_h\| \dd s.\]
By Gronwall's lemma $\sup_{s\le t }\|Y'_s-Y_s\| \le 0,$
which concludes the proof.
\end{proof}

\begin{lemma}\label{lem:integ}
Define $M_t = \Psi(t) \int_0^t \Psi(s)^{-1} U(s) \dd W_s,$ where $\Psi$ satisfies 
$\dd \Psi(s) = A(s) \Psi(s) \dd s$ and $\Psi(0)=\I.$ Assume $(T-s)\|U(s)\| \lesssim 1$ for $s\in [0,T)$. Assume that the assumptions of Lemma \ref{lem:SidedPsiBound} hold with $\eps_0 \in (0,1/2)$ and additionally
that there are constants $C_1, C_2 >0$ such that for all $0\le s<T$
\begin{equation}
\label{eq:A(s)} 
	\|A(s)\| \le C_1 \frac{1}{T-s} + C_2. 
\end{equation}
Then there exists an a.s. finite random variable $N$ such that for all $0\le s<T$ $\|M_s\| \le  (T-s)^{\eps_0-1} N.$
\end{lemma}

\begin{proof} 
Let $\ga \in (\eps_0,1/2)$ and define
\begin{equation}\label{barM}
M^{(\gamma)}_t = \int_0^t (T-s)^{1-\gamma}U(s) \dd W_s,
\end{equation}
so that $M_t = \int_0^t (T-s)^{\ga-1} \Psi(t) \Psi(s)^{-1} \dd M^{(\gamma)}_s$. 

By partial integration,
\[ M_t= (T-t)^{\ga-1} M^{(\gamma)}_t
 - \Psi(t) \int_0^t M^{(\gamma)}_s \dd\left( (T-s)^{\ga-1} \Psi(s)^{-1}\right).\]
By straightforward algebra
the integral appearing on the right-hand-side can be simplified and we get
\[ M_t=(T-t)^{\ga-1} M^{(\gamma)}_t - \Psi(t) \int_0^t M^{(\gamma)}_s(T-s)^{\gamma -2} \Psi(s)^{-1}  \left[(1-\gamma) \I  - (T-s)A(s)\right] \dd s.\]
By equation \eqref{eq:A(s)}, $\norm{ (1-\gamma) \I  - (T-s)A(s) } \le 1 + C_1 + C_2(T-s)$. 
Therefore,
\begin{multline*}
 \|M_t\|   \le (T-t)^{\ga-1} \|M^{(\gamma)}_t\| +\\ \sup_{0\le s\le t} \|M^{(\gamma)}_s\|  \int_0^t (T-s)^{\ga-2} \|\Psi(t)\Psi(s)^{-1}\| \left( 1 + C_1 + C_2{(T-s)}\right) \dd s .
\end{multline*}
Using Lemma \ref{lem:SidedPsiBound}, the integral on the right-hand-side of the preceding display can be bounded by a positive constant times
\[ \int_0^t (T-s)^{\ga-2} \left(\frac{T-s}{T-t}\right)^{1-\eps_0} \dd s
= \left({T-t}\right)^{-1+\eps_0} \int_0^t (T-s)^{-1 + \ga -\eps_0} \dd s.
\]
From the choice $\ga>\eps_0$, this last integral is bounded.
So we obtain 
$\|M_t\| \le (T-t)^{\eps_0-1} N $, 
with $N= C\sup_{0\le t \le T} \|M_t^{(\gamma)}\|$ for some $C>0$.
It remains to show that $N$ is a.s.\ finite.
By the assumption on $U$, the quadratic variation of $M^{(\gamma)}$ satisfies, since $\gamma < 1/2$, 
\[
\Big\|\qv{M^{(\gamma)}}_T \Big\| \le \int_0^T \frac1{(T-s)^{2\gamma}}\,\dd s < \infty.
\]
Hence, the result follows from the Dambis-Dubins-Schwarz theorem.
\end{proof}

\begin{lemma}\label{lem:improve}

Let $f\, :\, [0,T) \to [0,\infty)$ be nondecreasing and bounded on any subinterval $[0, \tau]$, $\tau < T$. Suppose $\rho$ is integrable, continuous and nonnegative on $[0,T)$. If 
$$\rho(t) \le f(t) +  C\int_0^t \left(\rho(s) + \rho^2(s)\right) \dd s, \quad t \in [0,T)$$ for some positive constant $C$,  
  then 
$\rho \lesssim  f$ on $[0,T)$.
\end{lemma}

For the proof we need the following Gronwall--Bellman type lemma.
A proof can be found in \cite{Mitrinovi} (Chapter XII.3, Theorem 4). 
\begin{lemma}\label{gronwall}
Let $\rho(t)$ be continuous and nonnegative on $[0,\tau]$ and satisfy 
\[\rho(t) \le f(t) + \int_0^t h(s) \rho(s) \dd s, \quad t \in [0,\tau],\]
where  $h$ is a nonnegative integrable function on $[0,T)$ and  with $f$ nonnegative, nondecreasting and bounded on $[0,\tau]$. Then 
\[ \rho(\tau) \le f(\tau)  \exp\left(\int_0^T h(s) \dd s \right) .\]

\end{lemma}
\begin{proof}[Proof of lemma \ref{lem:improve}]
Applying the  Gronwall--Bellman lemma with $h(s) = C(1 + \rho(s))$ gives that for any $\tau \in [0,T)$,
\[
\rho(\tau) \le   f(\tau)  \exp\left(\int_0^\tau h(s) \dd s\right) \le f(\tau)\exp\left(\int_0^T  C (1 + \rho(s)) \dd s\right). 
\]
The integral on the right-hand-side is finite. 
\end{proof}


\appendix
\section{Information projection and entropy method}\label{sec:appendix}
 
The following procedure to find the information projection is similar to the cross entropy method in rare event simulation. The algorithm proceeds by stochastic gradient descent to improve $\th$ using samples from proposals with a varying reference value for $\th$ (named $\th_n$ below), which is updated every $K$ steps.

\begin{algorithm}\label{alg:ce}\ 

{\bf Initialisation:} Choose a starting value for $\th$, let $n = 1$ and choose  decay weights $\alpha(n, k)$. 

{\bf Repeat} for $n = 1, 2, \dots$
\begin{enumerate}
\item {\bf Update $\th_n$.} Let $\th_n = \th$.
\item {\bf Sample proposals.}
 Sample $m = 1,\dots, M$ bridge proposals $X^{\circ(m)}$ with parameter $\th_n$.

\item {\bf Stochastic gradient descent.} 
 For $k = 1,  \dots, K$
\[\theta \leftarrow \theta - \alpha(n, k) \frac1M\sum_{m = 1}^{M} \frac{\dd P_\theta^\circ}{\dd P^\circ_{\theta_n}}(X^{\circ(m) }) \nabla_\theta \log \frac{\dd P^\circ_\theta}{\dd P^\star}(X^{\circ(m)}).    \]
\end{enumerate}
\end{algorithm}
If $M = 1$ and $K=1$ this an algorithm of stochastic gradient descent type and $\alpha_n = \alpha_0\frac{\gamma }{\gamma + n}$ would be a standard choice. But depending on the form of $\tilde b_\th$, the update in step $3$ might be computationally cheap in comparison with step $2$ and one would prefer to sample $M>1$ bridges in batches and do step $3$ for  $K>1$.

In figure \ref{fig:sinexample} we took starting value $\th = 0$, $\alpha_n =(10 +2n)^{-1}$ and $M=K=1$.

\bibliographystyle{harry}
\bibliography{lit}

   \end{document}